\documentclass[11pt]{article}
\usepackage{amssymb,amsmath,amsfonts,amsthm,mathtools,color}

\usepackage{mathrsfs}
\usepackage{bm} % for bm macro

\usepackage{comment} % for comment large section
\usepackage[title,titletoc,header]{appendix}
\usepackage{graphicx,subfig}

\usepackage{paralist}

\usepackage{tikz}
\usetikzlibrary{arrows,positioning,shapes.geometric}

\usepackage[ruled,vlined]{algorithm2e}

\graphicspath{ {./figures/} }

\usepackage[utf8]{inputenc}

\usepackage{indentfirst}
\usepackage{multicol}
\usepackage{booktabs}
\usepackage{url}
\usepackage[outdir=./]{epstopdf} %to enable the use of eps
\usepackage{stmaryrd}
\usepackage[shortlabels]{enumitem}
\usepackage{hyperref}
\hypersetup{
    colorlinks=true, %set true if you want colored links
    linktoc=all,     %set to all if you want both sections and subsections linked
    linkcolor=blue,  %choose some color if you want links to stand out
}
\numberwithin{equation}{section}

\newtheorem{Theorem}{Theorem}[section]
\newtheorem{Lemma}[Theorem]{Lemma}

\newtheorem{Assumption}{H.\!\!}

\theoremstyle{definition}
\newtheorem{Definition}{Definition}[section]

\newtheorem{Example}{Example}[section]

\theoremstyle{remark}
\newtheorem{Remark}{Remark}[section]

\def\to{\rightarrow}

\newcommand{\q}{\quad}   \newcommand{\qq}{\qquad}

\def\del{\partial} 
\def\eps{\varepsilon}

 \def\bs{\bigskip}

\def\bs{\boldsymbol}

\def\cA{\mathcal{A}}
\def\cB{\mathcal{B}}

\def\cD{\mathcal{D}}

\def\cF{\mathcal{F}}
\def\cG{\mathcal{G}}
\def\cH{\mathcal{H}}

\def\cJ{\mathcal{J}}

\def\cL{\mathcal{L}}
\def\cM{\mathcal{M}}

\def\cO{\mathcal{O}}

\def\cQ{\mathcal{Q}}

\def\cU{\mathcal{U}}

\def\d{{\mathrm{d}}}

\DeclareMathOperator*{\smaxld}{\textit{S}^{\lambda}_{\max}}

\DeclareMathOperator*{\Tr}{Tr}

\def\sE{{\mathbb{E}}}

\def\sI{{\mathbb{I}}}
\def\sN{{\mathbb{N}}}
\def\sP{\mathbb{P}}

\def\sR{{\mathbb R}}
\def\sS{{\mathbb{S}}}

\DeclareMathOperator*{\argmax}{arg\,max}
\DeclareMathOperator*{\argmin}{arg\,min}

\newcommand{\lc}
{\mathrel{\raise2pt\hbox{${\mathop<\limits_{\raise1pt\hbox
{\mbox{$\sim$}}}}$}}}

\newcommand{\gc}
{\mathrel{\raise2pt\hbox{${\mathop>\limits_{\raise1pt\hbox{\mbox{$\sim$}}}}$}}}

\newcommand{\ec}
{\mathrel{\raise2pt\hbox{${\mathop=\limits_{\raise1pt\hbox{\mbox{$\sim$}}}}$}}}

\def\bb{\begin{equation}} \def\ee{\end{equation}}
\def\bbn{\begin{equation*}} \def\een{\end{equation*}}

\def\beqn{\begin{eqnarray}}  \def\eqn{\end{eqnarray}}

\def\beqnx{\begin{eqnarray*}} \def\eqnx{\end{eqnarray*}}

\def\bn{\begin{enumerate}} \def\en{\end{enumerate}}

\def\bd{\begin{description}} \def\ed{\end{description}}



\begin{document}

\title{
Asymptotic Randomised Control with applications to bandits}

\author{
	Samuel N. Cohen\footnote{Mathematical Institute, University of Oxford and Alan Turing Institute,   \texttt{samuel.cohen@maths.ox.ac.uk}}
	\and
	Tanut Treetanthiploet\footnote{Alan Turing Institute,   \texttt{ttreetanthiploet@turing.ac.uk}}
}
\date{\today}

\maketitle
\begin{abstract}
We consider a general multi-armed bandit problem with correlated (and simple contextual and restless) elements, as a relaxed control problem. By introducing an entropy regularisation, we obtain a smooth asymptotic approximation to the value function. This yields a novel semi-index approximation of the optimal decision process. This semi-index can be interpreted as explicitly balancing an exploration--exploitation trade-off as in the optimistic (UCB) principle where the learning premium explicitly describes asymmetry of information available in the environment and non-linearity in the reward function. Performance of the resulting Asymptotic Randomised Control (ARC) algorithm compares favourably well with other approaches to correlated multi-armed bandits.

\smallskip


\end{abstract}

%\tableofcontents

%\medskip
 
\section{Introduction}
In many situations, one needs to decide between acting to reveal data about a system and acting to generate profit; this is the trade-off between exploration and exploitation. A simple situation where we face this trade-off is a multi-armed bandit problem, where one need to maximise the cumulative rewards obtained sequentially from playing a bandit with $K$ arm. The reward of an arm is generated from a fixed unknown distribution, which must be inferred 
`on-the-fly'. 

%The first theoretical results for the multi-armed bandit problem were proved by Gittins and Jones  \cite{Gittin_origin}. They formulated the bandit as an optimisation problem for a Markov Decision Process over an infinite time horizon with an underlying state corresponding to the posterior distribution.  An optimal solution is described in terms of an index strategy, where we always play an arm with the maximum index.  The crucial assumption guaranteeing the optimality of this index strategy is the independence between arms, which may not hold in more general settings.   Moreover, calculating such an index requires solving a non-standard optimal stopping theorem which is, in general, computationally expensive.

More recently,  multi-armed bandit problems are often formulated as a statistical problem (see e.g. \cite{adaptive_learning_survey,   Lattimore_book,  Tsitsiklis_correlated_bandit, GLM_bandit}), rather than an optimisation problem because of the computational cost.  Many novel algorithms for bandit problems are often proposed using heuristic justifications and are then shown to give theoretical guarantees in terms of a regret bound. Even though these constructive designed algorithms works well in some settings, it still fails to address a few issues which may appear in learning. This includes asymmetry of information, incomplete learning and effect of the horizon (see Section \ref{sec: limit}). 
%(either from a Bayesian or a frequentist perspective) 
%in symmetric settings which typically assume that rewards are the only observation and are linear in the unknown parameter. 
	
This paper aims to address these issues by considering the multi-armed bandit problem as an optimisation problem for a Markov Decision Process (MDP) over an infinite time horizon, as originally proposed by Gittins and Jones \cite{Gittin_origin}. The underlying state of the MDP corresponds to the posterior distribution which is an infinite state space and takes value in high dimension. Solving such an MDP is generally computationally infeasible. To overcome the computational issues, we modify the entropy regularised control formulation of Reisinger and Zhang \cite{John_relaxed_control} and Wang, Zariphopoulou and Zhou \cite{Relaxed_control} to study discrete time bandit setting and obtain an asymptotic approximation to the value function when the posterior variance is relatively small. This approximation results in a randomised index strategy which enjoys a natural interpretation as a sum of the instantaneous reward (exploiting) and the learning premium (exploring).

The main contribution of this work is to develop conceptual insight into a general class of learning (bandit) problems from first principles.  We aim to understand how to make decisions when information from each arm interacts with the others in asymmetric manner (i.e. when there is an arm which could be more informative than others), and also aim to understand the connection between reward,  observation and learning mechanism when rewards are non-linear in the unknown parameter. Our result inspires the design of a learning algorithm which performs well numerically with cheap computational costs. In this work, we do not aim to provide a regret analysis of the derived algorithm.

The paper proceeds as follows.  In Section \ref{sec: Idea ARC}, we describe how we  formulate various classes of bandit problems in terms of a discrete-time diffusion process.  By considering the diffusion dynamic in a regime with small uncertainty (e.g. small posterior variance), we propose the Asymptotic Randomised Control (ARC) algorithm together with a sketch derivation and summary of our main results. 
In Section \ref{sec: comparison},  we give an overview of well-known approaches for bandit problems and discuss some limitations of these algorithms under a specific setting, and how the ARC algorithm addresses these limitations.  We also discuss the connection between the resulting ARC algorithm and other approaches for bandit problems. In Section \ref{sec: Numerical},  we run a numerical experiment  to show that when the uncertainty is small, the ARC algorithm is numerically accurate, and draw comparisons with other approaches for bandit problems in different settings. Finally, in Section \ref{sec: Prove ARC}, we provide a formal derivation of the ARC approach where the further technical results are given in the appendix.  

\paragraph{Notation:} 
For any function $h: \sR^p \times \sR^q \to \sR^n$, we write $\del^k_m \del^l_d h$ for a tensor of degree $(1,k,l)$ with dimension $np^kq^l$  corresponding to the partial derivatives of $h$.  We write
    $\langle \cdot ; \cdot \rangle $ for Euclidean tensor products. 
    In particular,  $\langle P ; Q \rangle $ is a standard inner product when $P,Q$ are vectors and  $\langle P ; Q \rangle = \Tr(P^\top Q)$ when $P$ and $Q$ are square matrices. We also write $| \cdot |$ for the corresponding Euclidean norm. For any (Borel) random variables $X,Y$ and a $\sigma$-algebra $\cG$,  we write $\cL(Y|\cG)$ and  $\cL(Y|X, \cG)$ for the conditional law of $Y$ given $\cG$ and $\sigma(X) \vee \cG$, respectively.  We write $[K] := \{1, 2, ..., K\}$, $e_i$ for the $i$-th standard basis vector with an appropriate dimension,  $\bm{1}_n = (1, 1, ...,1) \in \sR^n$, and $I_n$ for the identity matrix with dimension $n$. Throughout the proof, we shall introduce a generic constant $C \geq 0$ to quantify an upper bound.  This constant does not depend on variables $(m,d, \lambda, t, T)$ introduced throughout the paper. 
% \begin{itemize}
	
% 	\item For any function $h: \sR^p \times \sR^q \to \sR^n$, we write $\del^k_m \del^l_d h$ for a tensor of degree $(1,k,l)$ with dimension $np^kq^m$  corresponding to the partial derivatives of $h$. 

%     \item We write
%     $\langle \cdot ; \cdot \rangle $ for Euclidean tensor products. 
%     In particular,  $\langle P ; Q \rangle $ is a standard inner product when $P,Q$ are vectors,  $\langle P ; Q \rangle = \Tr(P^\top Q)$ when $P$ and $Q$ are square matrices. We also write $| \cdot |$ for the corresponding Eucliean norm.
% 	\item For any (Borel) random variables $X,Y$ and a $\sigma$-algebra $\cG$,  we write $\cL(Y|\cG)$ and  $\cL(Y|X, \cG)$ for the conditional law of $Y$ given $\cG$ and $\sigma(X) \vee \cG$, respectively.  
% 	\item We write $[K] := \{1, 2, ..., K\}$.
% 	\item We write $e_i$ for the $i$-th standard basis vector with an appropriate dimension,  $\bm{1}_n = (1, 1, ...,1) \in \sR^n$, and $I_n$ for the identity matrix with dimension $n$.
% 	\item Throughout the proof, we shall introduce a generic constant $C \geq 0$ to quantify an upper bound.  This constant does not depend on variables $(m,d, \lambda, t, T)$ introduced throughout the paper. 
% \end{itemize}

\section{Asymptotic Randomised Control (ARC) approach}
\label{sec: Idea ARC}
This section gives an overview of the Asymptotic Randomised Control (ARC) method for a general class of learning problem with a sketch derivation. We defer a rigorous mathematical justification to Section \ref{sec: Prove ARC}. 

We first introduce a heuristic description of our learning (bandit) problem from Bayesian perspective.  Let $\bs{\theta} \sim \pi$ be an underlying (unknown) parameter with prior $\pi$ describing the rewards of our bandit system. When the $i$-th arm, of the $K$ arms, is chosen at time $t$, we observe a random variable $Y^{(i)}_t \sim \pi_i(\cdot | \bs{\theta})$ where is independent for each $t$ conditional on $\bs{\theta}$ and obtain a reward $r_i(Y^{(i)}_t)$\footnote{This particular structure allows our agents to observe more than just a reward.}. The objective of our agent is to find a sequence of decisions $(A_t)$ taking values in $[K]$ based on available information to maximise 
\begin{equation}
\label{eq: main objective}
 \sE_{\pi} \Big[\sum_{t=1}^\infty \beta^{t-1} r_{A_t} \big( Y^{(A_t)}_t \big) \Big] = \sE_{\pi} \bigg[ \sum_{t=1}^\infty \beta^{t-1}  \bigg(\sum_{i=1}^K \bm{1}(A_t = i) \sE_{\pi} \Big[ \sE_{\pi} \big[ r_{i} \big( Y^{(i)}_t \big) \big| \bs{\theta} \big] \Big| \cG^A_{t-1} \Big] \bigg) \bigg]
\end{equation}
where $\beta \in (0,1)$ and $\cG^A_t := \sigma(A_1, Y^{(A_1)}_1, ..., A_t, Y^{(A_t)}_t)$ is the $\sigma$-algebra corresponding to our observation up to time $t$.

We see from \eqref{eq: main objective} that $\sE_{\pi} \big[ \sE_{\pi} \big[ r_{i} \big( Y^{(i)}_t \big) \big| \bs{\theta} \big] \big| \cG^A_{t-1} \big]$ depends only on the posterior distribution of $\bs{\theta}$ at time $t-1$. Therefore, we can formulate the maximisation of \eqref{eq: main objective} as a Markov Decision Process (MDP) with a posterior distribution as an underlying state.  

\subsection{Bandit structures}
\label{sec: Example}
We first give a few examples of bandit structures where we can describe posterior distributions with a finite dimensional state process (taking values in an infinite state space).
\subsubsection{Gaussian distribution for correlated structure}
Suppose the prior $\pi$ for $\bs{\theta}$ is a multivariate normal $N(m, d)$ with dimension $p$ and the observation from the $i$-th bandit at time $t$ is given by a $q$-dimensional random vector $Y^{(i)}_t$ with  $\cL(Y^{(i)}_t| \bs{\theta}) \sim N \big(c_i^\top \bs{\theta},  P_i^{-1} \big)$  where $c_i \in \sR^{p \times q}$ and $P_i$ is a $q \times q$ diagonal matrix with non-negative entries.  
% Since $P_i$ is diagonal,  each component of $Y^{(i)}_t$ is independent (conditional on $\bs{\theta}$).  Each component of  $Y^{(i)}_t$ is a separate observation and  
We formally allow the $(j,j)$-th entry of $P_i$ to be $0$, if this is the case, we shall interpret that the $j$-th component of $Y^{(i)}_t$ is not observed.

For simplicity, we will demonstrate the calculation when all entries of $P_i$ are  positive.  The case with zero entries can be achieved either by simply taking the limit of the derived result or by considering each component separately.

Let $\cG^A_t$ denote the filtration describing our observation up to time $t$.  Suppose that our posterior at time $t$ is given by $\cL(\bs{\theta}| \mathcal{G}^A_t) = N\big(M_t, D_t\big)$.  By standard Bayesian inference,  we can show that  $$\cL(Y^{(i)}_{t+1} |  \mathcal{G}^A_t) = N\big(c_i^\top M_t,  c_i^\top D_t c_i + P_i^{-1} \big) \q \text{and} \q  \cL(\bs{\theta}|\mathcal{G}^A_{t+1}) = \cL(\bs{\theta}|Y^{(i)}_{t+1},  \mathcal{G}^A_t) = N\big(M_{t+1}, D_{t+1}\big)$$ where
$
M_{t+1} = \big(D_t^{-1} + c_iP_ic_i^\top\big)^{-1}\big(D_t^{-1} M_t + c_iP_iY^{(i)}_{t+1}\big) \q \text{and} \q
D_{t+1} = \big(D_t^{-1} + c_iP_ic_i^\top\big)^{-1}.
$

By the Woodbury identity, $(D_t^{-1} + c_i P_i c_i^T)^{-1} = D_t - D_tc_i(P_i^{-1} + c_i^\top D_tC_i)^{-1}c_i^\top D_t$, and hence
\begin{equation}
\label{eq: Gaussian mean propagation}
\left. \begin{aligned}
M_{t+1} - M_t&= \Big(D_t - D_t c_i(P_i^{-1} + c_i^\top D_t c_i)^{-1}c_i^\top D_t\Big) c_i P_i \Big( c_i^\top D_t c_i + P_i^{-1}\Big)^{1/2} Z^{(i)}_{t+1}, \\
D_{t+1} - D_t&= - D_t c_i(P_i^{-1} + c_i^\top D_t c_i)^{-1}c_i^\top D_t,
\end{aligned}
\right\}
\end{equation}
where $Z^{(i)}_{t+1} = ( c_i^\top D_t c_i + P_i^{-1})^{-1/2}(Y^{(i)}_{t+1} - c_i^\top M_t) \sim N(0, I_q)$.

Therefore, we can write the objective \eqref{eq: main objective} as an MDP with the objective
\begin{equation}
\label{eq: main objective Gaussian}
\sE_{m,d} \Big[\sum_{t=0}^\infty \beta^{t} f_{A _{t+1}}(M_t,  D_t)\Big] \q \text{with} \q f_i(m,d) = \int_{\mathbb{R}^q} r_i\big(c_i^\top m + (b^\top_id c_i + P^{-1}_i)^{1/2}z\big)\varphi(z)dz
\end{equation}
where $\varphi$ is the density of $N(0, I_q)$. The corresponding underlying state $(M_t,D_t)$ satisfying \eqref{eq: Gaussian mean propagation} is a discrete version of a diffusion process where the prior distribution is encoded as an initial state.

It is worth pointing out that this set-up covers various classes of multi-armed bandit problem which can be found in the literature:
\begin{Example}[Classical bandit]
\label{Ex: Gaussian classic bandit}
The case when $q = 1$,  $K = p$, $c_i =e_i$ and  $r_i(y) = y$ corresponds to the classical stochastic bandit which often finds in most of the literature (see e.g. \cite{Lattimore_book,  adaptive_learning_survey}).
\end{Example}
\begin{Example}[Linear bandit]
The case when $q=1$ and $r_i(y) = y$ corresponds to the linear bandit considered in Russo and Roy \cite{Russo_Roy_Thompson, IDS}. The case when $r_i$ is a non-linear function is considered in Filippi et al. \cite{GLM_bandit} with a different distribution assumption. 
\end{Example}
\begin{Example}[Structural bandit] 
The case when $p = q = 1$,  $c_i = 1$ and  $r_i(y) = u_i y + v_i$ corresponds to the structural bandit studied in Rusmevichientong et al.  \cite{Tsitsiklis_correlated_bandit}.
\end{Example}
\begin{Example}[Classical bandit with additional information]
\label{Ex: bandit adding info}
Suppose $q = p$,  $K = p$, $c_i = I_q$, $r_i(y) = \tilde{r}_i(y_i)$ and the $(i,i)$-th entry of $P_i$ is positive for all $i = 1, ..., K$.   In this case, the reward of the $i$-th arm is only associated with the parameter $\bs{\theta}_i$ (as $r_i(y)$ only depends on the $i$-th component of $y$). However,  when the $i$-th arm is chosen, we also observe some information associated to $\bs{\theta}_j$, provided that the $(j,j)$-th entry of $P_i$ is positive. We will use this example to demonstrate learning with asymmetric information in Section \ref{sec: address limit}.
\end{Example}

\begin{Example}[Semi-bandit feedback]
The case when $r_i(y) = \frac{1}{|\cJ_i|} \sum_{j \in [K]_i} y_j$ where $\cJ_i := \{ j : (P_i)_{jj} > 0 \} $ corresponds to the semi-bandit feedback considered in Russo and Roy \cite{Russo_Roy_Thompson, IDS}.
\end{Example}
\begin{Example}[Non-diagonal covariance structure]
\label{Ex: general linear}
Suppose that, when the $i$-th arm is chosen at time $t$,  we observe $\tilde{Y}^{(i)}_t|\bs{\theta} \sim N\big(\tilde{c}_i^\top \bs{\theta},  \tilde{P}_i^{-1} \big)$ with a corresponding reward function $\tilde{r}_i$, where  $ \tilde{P}_i$ is a positive-definite matrix.   By eigen-decomposition, we can write $\tilde{P}_i = Q_i P_i Q_i^\top$ where $Q_i$ is an orthogonal matrix and $P_i$ is a diagonal matrix with positive entries.  Define  $c_i = \tilde{c}_iQ_i$, ${Y}^{(i)}_t = Q_i^\top \tilde{Y}^{(i)}_t $ and $r_i(y) = \tilde{r}_i(Q_iy)$.  Since $Q_i$ is invertible,  observing $\tilde{Y}^{(i)}_t $ is equivalent to observing ${Y}^{(i)}_t $ and $r_i(Y^{(i)}_t) = \tilde{r}_i(\tilde{Y}^{(i)}_t)$. Thus, we recover the prescribed framework.
\end{Example}
\subsubsection{General distribution for independent structure}
Suppose that $\bs{\theta} = (\bs{\theta}_1, ..., \bs{\theta}_q)$ has a prior $\bigotimes_{j=1}^q \pi_j$ and when the $i$-th arm is chosen at time $t$, we observe
$Y^{(i)}_t := \big(Y^{(i,1)}_t, ..., Y^{(i,q)}_t \big) \sim \bigotimes_{j=1}^q \pi_i(\cdot | \bs{\theta}_j).$ Due to independence in the prior and observation,  the posterior of $\bs{\theta}$ also maintains independence of its components. Therefore, we can evaluate the posterior update of each $\bs{\theta}_j$ separately.  

For simplicity of discussion, we consider examples when $\bs{\theta}_j$ is one-dimensional.  The argument for more general cases can be easily extended to as long as $\pi_j$ and  $\pi_i(\cdot | \bs{\theta}_j)$ are a conjugate pair for all $i$ and $j$.  We implicitly allow $\pi_i(\cdot | \bs{\theta}_j)$ to be degenerate,  corresponding to the case when no information regarding  $\bs{\theta}_j$ is revealed when the $i$-th arm is chosen. 

In the following examples, we formulate the multi-armed bandit problem in terms of an MDP with an underlying discrete diffusion dynamic $(M,D)$, as  in \eqref{eq: Gaussian mean propagation}. The process $M$ is interpreted as an estimator of $\bs{\theta}$, while $D$ is the inverse precision of the estimate $M$ evolving in a deterministic manner. Here, we will only illustrate the corresponding dynamics of the underlying state $(M,D)$.  The objective function can then be written in the same manner as in \eqref{eq: main objective Gaussian}.
\begin{Example}[Binomial bandit] Suppose that the prior $\pi_j$ of $\bs{\theta}_j$ is $\text{Beta}(\alpha_j,  \beta_j)$ and  $\pi_i(\cdot | \bs{\theta}_j) = \text{Binomial}(n_i,  \bs{\theta}_j)$.  We again denote by $\cG^A_t$ our observations up to time $t$ and assume that $\cL(\bs{\theta}_j | \cG^A_t) = \text{Beta} \big( M^j_t/D_{j,t},  (1-M^j_t)/D_{j,t} \big)$.  When the $i$-th arm is chosen at time $t+1$, one can check that the posterior becomes $\cL(\bs{\theta}_j | \cG^A_{t+1}) = \text{Beta} \big( M_{j,t+1}/D_{j,t+1},  (1-M_{j,t+1})/D_{j,t+1} \big)$ with
% \begin{align*}
% M_{j,t+1} - M_{j,t} = \bigg(\frac{D_{j,t}}{1 +n_i D_{j,t}}\bigg)(Y^{(i,j)}_{t+1} - n_iM_{j,t}) \qq \text{and} \qq
% D_{j,t+1} - D_{j,t} = - \frac{n_i (D_{j,t})^2}{1 + n_i D_{j,t}}.
% \end{align*}
% We can show that $\sE[M_{j,t+1} - M_{j,t} | \cG^A_t] = 0$ and
% \begin{align*}
% \text{Var}[M_{j,t+1} - M_{j,t} | \cG^A_t] &= \bigg(\frac{D_{j,t}}{1 +n_i D_{j,t}}\bigg)^2\sE\big[n_i\bs{\theta}_j (1-\bs{\theta}_j) \big| \cG^A_t\big] = \bigg(\frac{D_{j,t}}{1 +n_i D_{j,t}}\bigg)^2 \bigg(\frac{n_iM_{j,t}(1-M_{j,t})}{1+D_{j,t}} \bigg) .
% \end{align*}
% Therefore, we can represent the dynamics of the posterior parameter of $\bs{\theta}_j$, when the $i$-th arm is chosen, by
{\small
\begin{equation}
\label{eq: Binomial dynamic}
\begin{aligned}
M_{j,t+1} - M_{j,t} = \bigg(\frac{D_{j,t}}{1 +n_i D_{j,t}}\bigg){\bigg(\frac{n_iM_{j,t}(1-M_{j,t})}{1+D_{j,t}} \bigg) }^{1/2} Z^{(i,j)}_{t+1} \q \text{and} \q
D_{j,t+1} - D_{j,t} = - \frac{n_i D_{j,t}^2}{1 + n_i D_{j,t}},
\end{aligned}
\end{equation}
}
where $Z^{(i,j)}_{t+1} := (Y^{(i,j)}_{t+1} - n_iM_{j,t})\Big(\frac{n_iM_{j,t}(1-M_{j,t})}{1+D_{j,t}} \Big)^{-1/2}$ satisfies $\sE Z^{(i,j)}_{t+1}  = 0, \text{Var}[Z^{(i,j)}_{t+1} | M_t, D_t] = 1$.  
\end{Example}
\begin{Example}[Poisson bandit]
Suppose that the prior $\pi_j$ of $\bs{\theta}_j$ is $\text{Gamma}(\alpha_j,  \beta_j)$ and  $\pi_i(\cdot | \bs{\theta}_j) = \text{Poisson}(n_i\bs{\theta}_j)$.  Assume that $\cL(\bs{\theta}_j | \cG^A_t) = \text{Gamma} \big( M_{j,t}/D_{j,t}, 1/D_{j,t} \big)$.  When the $i$-th arm is chosen at time $t+1$,  the posterior becomes $\cL(\bs{\theta}_j | \cG^A_{t+1}) = \text{Gamma} \big( M_{j,t+1}/D_{j,t+1}, 1/D_{j,t+1} \big)$ with
% \begin{align*}
% M_{j,t+1} - M_{j,t} = \bigg(\frac{D_{j,t}}{1 +n_i D_{j,t}}\bigg)(Y^{(i,j)}_{t+1} - n_i M_{j,t}) \qq \text{and} \qq
% D_{j,t+1} - D_{j,t} = - \frac{n_i (D_{j,t})^2}{1 + n_i D_{j,t}}.
% \end{align*}
%We can show that
%\begin{align*}
%\sE[M_{j,t+1} - M_{j,t} | \cG^A_t] &= 0 \\
%\text{Var}[M_{j,t+1} - M_{j,t} | \cG^A_t] &= \bigg(\frac{D_{j,t}}{1 +n_i D_{j,t}}%\bigg)^2\sE\big[n_i \bs{\theta}_j \big| \cG^A_t\big] = \bigg(\frac{D_{j,t}}{1 +n_i D_{j,t}}\bigg)^2 n_iM_{j,t}.
%\end{align*}

%  Similarly to above, we can represent the dynamics of the posterior parameter of $\bs{\theta}_j$, when the $i$-th arm is chos,en by
\begin{equation}
\label{eq: Poisson dynamic}
\begin{aligned}
M_{j,t+1} - M_{j,t} = \bigg(\frac{D_{j,t}}{1 +n_i D_{j,t}}\bigg)({n_i M_{j,t} })^{1/2} Z^{(i,j)}_{t+1} \q \text{and} \q
D_{j,t+1} - D_{j,t} = - \frac{n_i D_{j,t}^2}{1 + n_i D_{j,t}},
\end{aligned}
\end{equation}
where $Z^{(i,j)}_{t+1} := (Y^{(i,j)}_{t+1} - M_{j,t})(n_i M_{j,t} )^{-1/2}$ satisfies $\sE Z^{(i,j)}_{t+1}  = 0, \text{Var}[Z^{(i,j)}_{t+1} | M_t, D_t] = 1$.  
\end{Example}
\subsection{From multi-armed bandit to diffusive control problem}
\label{sec: setup}
Inspired by the dynamics \eqref{eq: Gaussian mean propagation}, \eqref{eq: Binomial dynamic} and \eqref{eq: Poisson dynamic},   we introduce a general framework of a discrete time diffusion control problem which can be used to studied the multi-armed bandit.  

Let $(\Omega, \mathbb{P}, \mathcal{F})$ be a probability space equipped with two independent IID sequences of $U[0,1]$ random variables $(\bs{\xi}_t)_{t \in \sN}$ and $(\zeta_t)_{t \in \sN}$. We define the filtration $\mathcal{F}_t := \sigma(\bs{\xi}_s, \zeta_s : s \leq t)$. Here, $(\zeta_t)$ represents a random seed used to select a random decision in each time step, whereas  $(\bs{\xi}_t)$ represents the randomness of the outcome. 

At time $t$, our agent chooses an action $A_t$ taking values in $[K]$ produced by randomly selecting from a probability simplex $U_t$ where $(U_t)_{t \in \sN}$ is $(\cF_t)$-predictable. 
% $(A_t)_{t \in \mathbb{N}}$ taking values in a finite set $[K]$ where, for each $t$, $A_t$ is measurable with respect to $\sigma(\mathcal{F}_{t-1}, \zeta_t)$. Here, our agent does not choose $(A_t)_{t \in \mathbb{N}}$ directly but instead chooses a relaxed control, that is, a family of probability vectors $U_t \in \Delta^K := \big\{ u \in [0,1]^K : \sum_{i=1}^K u_i = 1 \big\}$ corresponding to the conditional law of $A_t$, i.e. $U_t := \mathbb{E}[e_{A_t} | \mathcal{F}_{t-1}].$ 
Without loss of generality, we assume that $A_t = \cA(U_t, \zeta_t)$ where $
    \cA(u, \zeta) = \sup\big\{i : \sum_{k=1}^i u_k \geq \zeta \big\},
$

Let $\Psi$ be a collection of measurable functions $\big\{ \varphi : \sN \times \Theta \times \cD \to \Delta^K \big\}$ describing our feedback actions. For each $\varphi \in \Psi$, we define the corresponding posterior dynamics by $(M^{\varphi, m, d}, D^{\varphi, m, d})$ taking values in $\Theta \times \cD \subseteq \sR^p \times \sR^q$ by $(M^{\varphi, m, d}_0, D^{\varphi, m, d}_0) = (m,d)$ and
\begin{align}
    \label{eq: feedback posterior dynamic}
    (M^{\varphi, m, d}_{t}, D^{\varphi, m, d}_{t}) =  \Phi\Big(M^{\varphi, m, d}_{t-1}, D^{\varphi, m, d}_{t-1},\cA\big(\varphi(t, M^{\varphi, m, d}_{t-1}, D^{\varphi, m, d}_{t-1}), \zeta_{t} \big) , \bs{\xi}_{t} \Big).
\end{align}
% The state of our system is represented by a Markov process $(M,D)$ taking values in $\Theta \times \cD \subseteq \sR^p \times \sR^q$.   Suppose that, when the control $i \in \mathcal{A}$ is chosen, the underlying state evolves according to the transition map 
where a measurable function
$\Phi : \Theta \times \cD \times [K] \times [0,1] \to \Theta \times \cD$ is given in a form
\begin{equation}
\label{eq: transition map}
\Phi(m,d, i, \xi) := 
\begin{pmatrix}
m \\ d
\end{pmatrix} +
\begin{pmatrix}
\mu_i(m,d) 
 \\ b_i(m,d) 
\end{pmatrix} +
\begin{pmatrix}
\sigma_i(m,d) z(m,d,i, \xi)\\0
\end{pmatrix} 
\end{equation}
with $\mu_i :  \Theta \times \cD \to \sR^p$,  $b_i:  \Theta \times \cD  \to \sR^q $, $\sigma_i :  \Theta \times \cD  \to \sR^{p \times r}$,  $ z :  \Theta \times \cD \times [K]  \times [0,1] \to \sR^r$. 

We also write $U^{\varphi, m, d}_{t} := \varphi(t, M^{\varphi, m, d}_{t-1}, D^{\varphi, m, d}_{t-1})$ and $A^{\varphi, m, d}_{t} := \cA(U^{\varphi, m, d}_{t}, \zeta_t)$. For notational simplicity, when clear from context, we will indicate $(m,d)$ as a subscript in the probability measure  and omit the superscript $(m,d)$ on $(M,D)$.

As illustrated in \eqref{eq: main objective} and elaborated in \eqref{eq: main objective Gaussian},  the objective of our control problem is to solve
\begin{equation}
\label{Eq: no entropy reward}
\left.
\begin{aligned}
 V(m,d) &:=
 \sup_{\varphi \in \Psi}V^\varphi(m,d), \qquad \text{where} \\
V^\varphi(m,d) &:= \mathbb{E}_{m,d}\Big[\sum_{t=0}^{\infty}\beta^{t} f_{A^{\varphi}_{t+1}}\big(M^{\varphi}_t, D^{\varphi}_t\big)\Big] =\mathbb{E}_{m,d}\Big[\sum_{t=0}^{\infty}\beta^{t} \Big(\sum_{i=1}^K f_{i}\big(M^{\varphi}_t, D^{\varphi}_t \big) U^\varphi_{i,t+1}\Big)\Big], 
\end{aligned}
\right\}
\end{equation}
 where $f:  \Theta \times \cD  \to \sR^K$. We assume that $f$ satisfies the following assumption.

\begin{Assumption}
\label{Assump: Instantaneous cost}
$f$ is $3$-times differentiable, and there exists a constant $C \geq 0$ such that for any $k,l \in \sN \cup \{0\}$ with $k+l \in \{1,2,3\}$, $|\partial_m^k \partial_d^l f|  \leq C. $
\end{Assumption}

% We shall interpret the process $(M^\varphi_t)$ $m$ as an estimator of the parameters in our model, while $d$ represents the variance, or \textit{inverse precision}, of our estimates. (The decomposition of the underlying state considered here is implicitly referred to as the `knowledge state' in the Knowledge Gradient (KG) considered by Ryzhov et al. \cite{Knowledge_Gradient}.)

Inspired by \eqref{eq: Gaussian mean propagation}, \eqref{eq: Binomial dynamic} and \eqref{eq: Poisson dynamic},  we also make the following assumption on the state dynamics
\begin{Assumption} 
\label{Assump: Dynamic Asumption} $\cD $ is a compact set and
there exists a constant $C > 0$ and a norm $\|\cdot \|$ on $\cD$  such that
\begin{enumerate}[(i)]
\item For any $i \in [K]$ and $(m,d) \in  \Theta \times \cD$,  $\|d +   b_i(m,d)\| \leq  \|d\|$. 
\item For any $i \in [K]$, $\xi \in [0,1]$ and $(m,d) \in  \Theta \times \cD$,  $m + \mu_i(m,d) + \sigma_i(m,d) z(m,d, i, \xi) \in  \Theta$. 
\item For any $i \in [K]$, $(m,d) \in  \Theta \times \cD$ and $t \in \sN$, $\sE[z(m,d, i, \bs{\xi}_t)] =0$, $\text{Var}[z(m,d,i, \bs{\xi}_t)] =I_r$.
\item For any $i \in [K]$,  $(m,d) \in  \Theta \times \cD$ and $t \in \sN$,  $\sE\big[ \big|z(m,d, i,\bs{\xi}_t) \big|^3 \big] \leq C.$
\item For any $\psi \in \big\{b_i, \mu_i, \big(\sigma_i\sigma_i^\top\big)  :i \in \mathcal{A}\big\}$, and $(m,d) \in  \Theta \times \cD$, we have
$$\sup_{m \in  \Theta} |\psi(m,d)|\leq C\|d\|^2,  \quad \sup_{m \in  \Theta}|\partial_m \psi(m,d)| \leq C\|d\|^2, \q \text{and} \q \sup_{m \in  \Theta}|\partial_d \psi(m,d)| \leq C\|d\|.$$
\end{enumerate}
\end{Assumption}
Conceptually, (H.\ref{Assump: Dynamic Asumption})$(i)$ says that our precision (our knowledge) of the parameter always improves with more observation. (H.\ref{Assump: Dynamic Asumption})$(ii)$ says that the updated parameter estimate always lies in our parameter set $\Theta$. (H.\ref{Assump: Dynamic Asumption})$(iii)-(iv)$ are the structural assumptions allowing to express our results in terms of $\mu_i$ and $\sigma_i$. Finally, (H.\ref{Assump: Dynamic Asumption})$(v)$ appears naturally through the propagation of the information as discussed in \eqref{eq: Gaussian mean propagation}, \eqref{eq: Binomial dynamic} and \eqref{eq: Poisson dynamic}. (H.\ref{Assump: Dynamic Asumption})$(v)$ is the key assumption for the asymptotic analysis that will be considered in the later section.

\begin{Remark}
It is worth emphasising that the above set-up covers examples discussed in Section \ref{sec: Example}. In particular, if the expected reward $\sE[r_i(Y^{(i)}_t)|\bs{\theta}]$ is linear in $\bs{\theta}$ and $M_t := \sE_\pi[\bs{\theta}|\cF^A_t]$, then the corresponding function $f$ is linear in $m$ and does not depend on $d$; consequently $(H.\ref{Assump: Instantaneous cost})$ becomes trivial. Furthermore, we also see that the dynamics \eqref{eq: Gaussian mean propagation}, \eqref{eq: Binomial dynamic} and \eqref{eq: Poisson dynamic}  can be written in the form \eqref{eq: transition map} and satisfy (H.\ref{Assump: Dynamic Asumption})$(i)$-(H.\ref{Assump: Dynamic Asumption})$(iv)$, with appropriate norms and with $ \Theta = \sR^p$,  $ \Theta = [0,1]^p$ and $ \Theta = [0,\infty)^p$, respectively. 
% Here, we see that (H.\ref{Assump: Dynamic Asumption}) holds for \eqref{eq: Gaussian mean propagation}  with operator norm on a collection of positive definite matrix whereas (H.\ref{Assump: Dynamic Asumption}) holds for \eqref{eq: Binomial dynamic} and \eqref{eq: Poisson dynamic} with a standard Euclidean norm. 
Moreover, we can see that \eqref{eq: Gaussian mean propagation} and \eqref{eq: Binomial dynamic} also satisfy (H.\ref{Assump: Dynamic Asumption})$(v)$. Unfortunately,  $\eqref{eq: Poisson dynamic}$ does not satisfy (H.\ref{Assump: Dynamic Asumption})$(v)$ since $\sigma_i(m,d)$ is unbounded. We may overcome this issue by replacing $\sigma_i$ with its smooth truncation to ensures that (H.\ref{Assump: Dynamic Asumption})$(v)$ holds and consider this as an approximation to the original Poisson bandit problem.  
% We shall interpret this as an approximate dynamic for a Poisson setting. 
\end{Remark}

\begin{Remark}
We may modify (H.\ref{Assump: Dynamic Asumption}) to consider an ergodic diffusion with small perturbation which is closely related to the Kalman-filtering theory.  We state the corresponding assumption here for precision of our discussion.  Nonetheless, we will focus on (H.\ref{Assump: Dynamic Asumption}) for clarity, and  discuss how to extend our analysis to this framework in Remark \ref{rem:Kalman} and Theorem \ref{Thm:Kalman}.
\end{Remark}
\begin{Assumption} 
\label{Assump: Kalman}
(H.\ref{Assump: Dynamic Asumption}) holds with $(i)$ replaced by
\begin{enumerate}
\item[$(i')$]  For any $i \in [K]$ and $(m,d) \in  \Theta \times \cD$,  $\|d +   b_i(m,d)\| \in \cD$.
\end{enumerate}
\end{Assumption}

%In addition, we also find in the classical Gaussian frameworks,  $\mu_i , b_i$ and $\sigma_i$ does not depend on $m$ for all $i \in [K]$. 
%We also consider this as a special case of our problem which reduces our approximation error.
%\begin{Assumption}
%\label{Assump: Linear}
%$f$ is linear in $m$ and does not depend on $d$. 
%Moreover,  $\mu_i , b_i$ and $\sigma_i$ does not depend on $m$ for all $i \in [K]$.
%\end{Assumption}

\subsection{Learning premia and index strategies (Inspiration of the ARC)}
\label{sec:learning premium}
Solving \eqref{Eq: no entropy reward} requires us to solve the Bellman equation with $(M^\varphi, D^\varphi)$ as an underlying state which is generally computationally intractable. Thanks to the nature of learning problems, we see from (H.\ref{Assump: Dynamic Asumption}) that the state $(M^\varphi, D^\varphi)$ (which corresponds to the posterior parameters) will not change much when $\|D^\varphi\|$ is small.  
% Therefore, it is reasonable to consider an approximate solution to \eqref{Eq: no entropy reward} in the asymptotic regime when $\|d\|$ is small. 
In this section, we explore the intuition behind a few learning approaches, which inspire us to study the second order asymptotic expansion of \eqref{Eq: no entropy reward}  over small $\|d\|$ and obtain the ARC algorithm.

Consider the classical Gaussian bandit as discussed in Example \ref{Ex: Gaussian classic bandit}, which corresponds to the case $f_i(m,d) =m_i$,  $\Theta = \sR^K$, and $\cD$ is a family of diagonal matrices with positive entries. In this setting, Gittins \cite{Gittins_book} shows that the optimal solution to \eqref{Eq: no entropy reward} is given in terms of an index strategy where the agent always chooses an arm to maximise an index $\alpha : \Theta \times \cD \to \sR^K$ where $\alpha_i(m,d) = m_i + \ell(\beta, d_{ii})$ with $\ell(\beta, d_{ii})$ increases in $d_{ii}$ (see also \cite{Brezzi_and_Lai_approx, Note_Gittins_UCB} for an approximation). In short, we see that the index $\alpha$ is a sum of two terms:
\begin{equation}
\label{eq: index decomposition}
\alpha = \text{Exploitation gain} + \text{Learning premium}.
\end{equation}
where the exploitation gain describes the expected reward that the agent will obtain given the current estimate whereas the learning premium describes the benefit for learning. 

 This observation inspires an \textit{optimistic principle} to design learning algorithms using an upper-confidence bound (UCB) (see e.g. \cite{ UCB_tuned, Bayes_UCB, GLM_bandit}).  Roughly speaking, the UCB approach chooses an arm based on an index of the form \eqref{eq: index decomposition} where the learning premium scales with uncertainty of the estimate reward. Since we add uncertainty as an additional reward, one could interpret this as a claim that the agent should have a preference for uncertain choices (in order to encourage learning). This is a misleading conclusion as the following example shows.
\begin{Example}[Uncertainty Preference] 
\label{ex: uncertain} Let consider a bandit with two arms.  Suppose that the reward of the first arm is sampled from $N(\bs{\theta}, 1)$ where $\bs{\theta}$ is not known, while the reward of the second arm is fixed and always $1$.  Suppose that we only collect the reward of the arm that we choose, but we always observe the reward of the first arm. Hence, we do not have to play the first arm to learn $\bs{\theta}$. Thus, most decision makers (without taking any risk/ uncertainty aversion) will choose arms purely based on the estimate $m$ of $\bs{\theta}$. In particular,  they will choose the first arm if the estimate  $m>1$ and choose the second arm otherwise.  
\end{Example}

In the above example,  we see that the reward of the first arm is more uncertain than the second arm, but a preference for uncertainty does not benefit our decision\footnote{In fact, in many behavioural models,  people display a bias (pessimism) against risk and uncertainty (see e.g. \cite{DR_original, Keynes, Knight} for general settings and \cite{Robust_Gittins_Our_paper} for the classical bandit setting). In practice, many decision makers still prefer the second arm in Example \ref{ex: uncertain} even if $m > 1$ to avoid risky and uncertain outcomes.}.  A better interpretation of the optimistic principle is that we have a preference for information gain or the `reduction' in uncertainty (rather than the uncertainty itself).   In the case of classical bandit, the information gain corresponds to the uncertainty of the estimate which suggests the misleading conclusion that we prefer an uncertain choice. 

The above observation leads us to ask, `how we can quantify information gain?' Consider a simple greedy approach, where we always stick with the best policy from the estimate at time $0$, i.e. for a fix $m \in \Theta$, let $\varphi^{GD}(t,m',d') = e_k$  where $k = \argmax_{i} m_i$.  Since $\bs{\theta} \sim N(m,d)$,
\begin{equation}
\label{eq: Greedy lost}
 V(m,d) - V^{\varphi^{GD}}(m,d) \leq \sum_{t=0}^\infty \beta^t \sE\big(\max_{i} \bs{\theta}_i - \max_{i} m_i )\leq  (1-\beta)^{-1}\sqrt{\|d\|\log K  },
\end{equation}
the first inequality follows from the fact that we cannot make a better decision than the best decision with known $\bs{\theta}$, and the second inequality follows from the Gaussian maximal inequality (see e.g. Chernozhukov et al. \cite[Theorem 1]{Gaussian_maximal}). 

The greedy strategy could be interpreted as a first order approximation to the optimal solution, considering only Exploitation gain in \eqref{eq: index decomposition}. This decision introduces error (relative to the best strategy) with order $\cO(\|d\|^{1/2})$.  This perspective suggests an index (like) strategy, where we choose the learning premium as a second order approximation (with respect to $d$).

Unfortunately,  as discussed in Reisinger and Zhang \cite{John_relaxed_control} in a continuous time setting, the solution to \eqref{Eq: no entropy reward} is, in general, non-smooth and thus difficult to obtain asymptotic approximation. To overcome this difficulty, we introduce an entropy regularisation (see also \cite{Relaxed_control, Gradient_flow_lukasz}) to analyse \eqref{Eq: no entropy reward} when $\|d\|$ is small. We later show that this approximation (and its corresponding strategy) introduces error with order $\cO(\|d\|)$, compared to $\cO(\|d\|^{1/2})$ for the naive greedy approach.
\subsection{From a regularised control problem to ARC algorithm}
\label{sec: core derivation ARC}
In this section, we will use an entropy regularised control to construct an approximation to the value function \eqref{Eq: no entropy reward} and sketch a derivation of the corresponding ARC algorithm to solve the diffusive control problem, introduced in Section \ref{sec: setup}.  The proof of the error of the approximation will be provided in Section \ref{sec: Prove ARC}.

We first observe that for $a \in \sR^K$, we can approximate the maximum function by
\begin{equation}
\label{eq: smax}
\max_i a_i \approx \smaxld(a) := \sup_{u \in \Delta^K} \bigg(\sum_{i=1}^K u_i a_i + \lambda\mathcal{H}(u)\bigg)
\end{equation}
where $\cH$ is a smooth (entropy) function and $\lambda > 0$ is a small regularisation parameter.  We also observe that a standard corollary to Fenchel's inequality (see e.g. Rockafellar \cite{Convex_Analysis_book}) yields
\begin{equation}
\label{eq: nu}
\nu^\lambda(a) := \partial_a  \smaxld(a) =   \argmax_{u \in \Delta^K}\bigg(\sum_{i=1}^K u_i a_i + \lambda\mathcal{H}(u)\bigg).
\end{equation}

\begin{Remark}[Shannon Entropy]
\label{Rem: Shannon Entropy}
For simplicity of discussion, we defer the formal assumptions on $\cH$ to Definition \ref{def: smooth entropy} in Section \ref{sec: Prove ARC}. The reader may simply suppose $\cH$ is the Shannon entropy, $\mathcal{H} (u):= -\sum_{i=1}^K u_i \ln u_i$. In this case, we have $\smaxld(a) =  \lambda \ln\big(\sum_{i=1}^K\exp\big(a_i/\lambda\big)\big)$ and  
$\nu^\lambda_i(a) = \exp\big(a_i/\lambda\big)\big/\big(\sum_{j=1}^K\exp\big(a_j/\lambda\big)\big)$.
\end{Remark}

% By introducing an entropy function,  we obtain a smooth approximation of the maximum $\smaxld(a)$, where $\smaxld(a) \to \max_i a_i $ as $\lambda \downarrow 0$ uniformly (Theorem \ref{thm: equivalence to boundedness}).  
\begin{Definition}
	Let $\mathcal{H}: \Delta^K \to \mathbb{R}$ be a smooth entropy (Definition \ref{def: smooth entropy})  and $\lambda > 0$. Define
	\begin{equation}
	\label{eq: infinite horizon value}
	\left.
	\begin{aligned}
	V^\lambda_\infty(m,d) &:= \sup_{\varphi \in \Psi} V^{\lambda,\varphi}_\infty(m,d), \qquad \text{where} \\
	V^{\lambda,\varphi}_\infty(m,d) &:= \mathbb{E}_{m,d}\Bigg[\sum_{t=0}^{\infty}\beta^{t} \bigg( \Big(\sum_{i=1}^K f_{i}\big(M^{\varphi}_t, D^{\varphi}_t \big) U^\varphi_{i, t+1}\Big) + \lambda \mathcal{H}(U^\varphi_{t+1}) \bigg)\Bigg].
\end{aligned}	
\right\}
	\end{equation}
\end{Definition}
As inspired by the discussion in Section \ref{sec:learning premium}, we would like to construct a function $\alpha^\lambda :  \Theta \times \cD  \to \sR^K$ where the $i$-th component of  $\alpha^\lambda$ corresponds to an `incremental reward' over a single step for choosing the $i$-th option.  In particular, let us assume that when $\|d\|$ is sufficiently small, the value function \eqref{eq: infinite horizon value} is approximated by the maximum amongst the values of the available options, i.e.,
\begin{equation}
\label{eq: approx V}
V^\lambda_\infty(m,d) \approx (1-\beta)^{-1}\big( \smaxld \circ \alpha^\lambda \big) (m,d).
\end{equation}

By the dynamic programming principle and the dynamic \eqref{eq: feedback posterior dynamic}-\eqref{eq: transition map}, we can rewrite \eqref{eq: infinite horizon value} as 
{\small
\begin{align}
V^\lambda_{\infty}(m,d) &= \sup_{u \in \Delta^K}\bigg(\sum_{i=1}^K\bigg(u_{i}\Big(f_i(m,d) + \beta \mathbb{E}[V^\lambda_{\infty}(\Phi(m,d, i, \bs{\xi}))]\Big) \bigg) + \lambda \mathcal{H}(u)\bigg) \\
& \approx \sup_{u \in \Delta^K}\bigg(\sum_{i=1}^K\bigg(u_{i}\Big(f_i(m,d) + \beta(1-\beta)^{-1} \mathbb{E}\Big[\big( \smaxld \circ \alpha^\lambda \big) \big(\Phi(m,d, i, \bs{\xi})\big)\Big]\Big) \bigg) + \lambda \mathcal{H}(u)\bigg) \label{eq: infinite horizon DPP}. 
\end{align}
}
Using the approximation \eqref{eq: approx V} on the LHS and applying \eqref{eq: smax}, we obtain from \eqref{eq: infinite horizon DPP} that
{\small
\begin{align*}
\big( \smaxld \circ \alpha^\lambda \big) (m,d)&\approx \smaxld  \bigg( f(m,d) + \beta(1-\beta)^{-1} \mathbb{E}\Big[\big( \smaxld \circ \alpha^\lambda \big) \big(\Phi(m,d,  \cdot, \bs{\xi})\big) - \big( \smaxld \circ \alpha^\lambda \big) (m,d) \bm{1}_K\Big]\bigg).
\end{align*}
}
 In particular, this suggests the approximation
\begin{equation}
\label{eq: approximate alpha}
\alpha^\lambda(m,d) \approx f(m,d) + \beta(1-\beta)^{-1} \mathbb{E}\Big[\big( \smaxld \circ \alpha^\lambda \big) \big(\Phi(m,d,  \cdot, \bs{\xi})\big) - \big( \smaxld \circ \alpha^\lambda \big) (m,d)\bm{1}_K \Big].
\end{equation}

Since our dynamics form a discrete diffusion process,  a discrete-time version of Ito’s lemma suggests that we can choose
\begin{equation}
\label{eq: alpha}
\alpha^\lambda(m,d) := f(m,d) + \beta(1-\beta)^{-1} L^\lambda ( m ,d ),
\end{equation}
where $L^\lambda :  \Theta \times \cD  \to \sR^K$  is given by
\begin{equation}
\label{eq: L expression}
\begin{aligned}
L^\lambda_i(m,d) &:= \langle \mathcal{B}^\lambda(m,d) ; b_i(m, d) \rangle + \langle \mathcal{M}^\lambda(m,d) ; \mu_i(m,d) \rangle +  \frac{1}{2}\langle \Sigma^\lambda(m,d) ; \sigma_i\sigma_i^\top(m,d) \rangle.
\end{aligned}
\end{equation}
Here, 	 $\mathcal{B}^{\lambda} : \Theta \times \cD  \to \sR^p$,  $\mathcal{M}^{\lambda} : \Theta \times \cD  \to \cD $ and $\Sigma^{\lambda} : \Theta \times \cD  \to \sR^{p \times p}$,
{\footnotesize
\begin{equation}
\label{eq: B, M, Sigma Expression}    
\left.
\begin{aligned}
	&\mathcal{B}^{\lambda}(m, d) := \sum_{j =1}^K \nu^{\lambda}_j\big((f(m,d) \big) \del_d f_j(m,d), \qq
	\mathcal{M}^{\lambda}(m, d) := \sum_{j =1}^K \nu^{\lambda}_j\big((f(m,d) \big) \del_m f_j(m,d),  \\
	&\Sigma^{\lambda}(m,d) :=\sum_{j =1}^K \Big(\nu^{\lambda}_j\big((f(m,d) \big)\del^2_m f_j(m,d)\Big) + \frac{1}{\lambda}\sum_{i,j =1}^K\Big(\eta^{\lambda}_{ij}\big((f(m,d) \big) \big(\del_m f_i(m,d) \big) \big(\del_m f_j(m,d) \big)^\top\Big). 
	\end{aligned}
	\right\}
\end{equation}	
}
where $\eta^\lambda(a) := \lambda \del_a^2S^\lambda_{\max}(a)$ (which is determined by $\eta_{ij}^\lambda(a) = \nu^\lambda_i(a)(\mathbb{I}(i=j) - \nu^\lambda_j(a) )$ for the case of Shannon Entropy given in Remark \ref{Rem: Shannon Entropy}).   
% 	It is worth noting that the terms $\mathcal{B}^{\lambda}$,  $\mathcal{M}^{\lambda}$ and $\Sigma^{\lambda}$ are derivatives $\del_d, \del_m$ and $\del^2_m$ of  $\smaxld \circ f$.

Moreover, by substituting \eqref{eq: approximate alpha} into \eqref{eq: infinite horizon DPP}, we can write 
\begin{align}
\label{eq: V regularise approximation}
V^\lambda_{\infty}(m,d) 
& \approx \sup_{u \in \Delta^K}\bigg(\sum_{i=1}^K u_{i} \alpha^\lambda_i(m,d)+ \lambda \mathcal{H}(u)\bigg) +\beta(1-\beta)^{-1}\big( \smaxld \circ \alpha^\lambda \big) (m,d),
\end{align}
with a small error for small $\|d\|$. This suggests an approximate solution ${u}^{*,\lambda}(m,d) \approx \nu^\lambda \big( \alpha^\lambda(m,d)\big)$ where the agent choose an arm randomly based on the softmax of the index $ \alpha^\lambda(m,d)$.

\subsection{Description of the main results}
In earlier section, we introduce an entropy regularisation and obtain an approximation of $V^\lambda_\infty$.  Unfortunately,  derivatives of the smooth approximation explode as $\lambda \to 0$.  This means that approximation error \eqref{eq: V regularise approximation} may explode if we take $\lambda \to 0$ while fixing $\|d\|$. Therefore, we will quantify errors in terms of $\|d\|$ and $\lambda^{-1}$.   

\begin{Theorem}[Error bound in the regularised control problem]
	\label{Thm: main bound for value function and approximation}
Suppose that (H.\ref{Assump: Instantaneous cost}) and (H.\ref{Assump: Dynamic Asumption})  holds.  There exists a constant $C \geq 0$ such that for any   $(m,d) \in \Theta \times \cD$ and $\lambda >0$
	\begin{align*}
&\big| V^\lambda_{\infty}(m,d)  -(1-\beta)^{-1}\big(\smaxld \circ \alpha^\lambda\big)(m,d) \big| \leq  C \|d\|^3\big(1 + \lambda^{-2} + \lambda^{-3} \|d\|  \big).
\end{align*}
and 
\begin{align*}
V^\lambda_{\infty}(m,d)   - V^{\lambda,\varphi^\lambda}_\infty(m,d) \leq C \|d\|^3\big(1 + \lambda^{-2} + \lambda^{-3} \|d\|  \big)
\end{align*}
where 
$V^\lambda_{\infty}$  is the optimal regularised value function and $ V^{\lambda, \varphi^{{\lambda}}}$ is the regularised value function corresponding to the feedback policy $\varphi(t, m,d) \mapsto \nu^\lambda \big( \alpha^\lambda(m,d)\big)$  (cf. \eqref{eq: infinite horizon value}).
\end{Theorem}

 Introducing the entropy regularisation in \eqref{eq: infinite horizon value} yields an error $\cO(\lambda)$ while error bounds for our approximation (Theorem \ref{Thm: main bound for value function and approximation}) explodes as $ \lambda \to 0$. To solve the unregularised problem \eqref{Eq: no entropy reward}, we choose the regularised parameter $\bs{\lambda}$ as a function of $(m,d) \in \Theta \times \cD$ to balance the trade-off between the regularisation error and the approximation error.
 
 \begin{Theorem}[Error bound in the unregularised control problem]
	\label{Thm: error bound consistent}
Suppose that (H.\ref{Assump: Instantaneous cost}) and (H.\ref{Assump: Dynamic Asumption}) hold and let $\bs{\lambda} : \Theta \times \cD \to (0,\infty)$ be such that  $\underline{c} \|d\|^\kappa \leq \bs{\lambda}(m,d) \leq \bar{c}\|d\|^\kappa$ for some constant $\underline{c}, \bar{c}, \kappa > 0$.  There exists a constant $C \geq 0$ such that for any   $(m,d) \in \Theta \times \cD$
\begin{equation}
\label{eq: sup-optimal consistent}
V(m,d) - V^{\varphi^{\bs{\lambda}}}(m,d) \leq C ( \|d\|^{3-2\kappa} + \|d\|^{4-3\kappa} + \|d\|^{\kappa} )
\end{equation}
where 
$V$  is the optimal (unregularised) value function and $ V^{ \varphi^{\bs{\lambda}}}$ is the (unregularised) value function corresponding to the feedback policy $\varphi^{\bs{\lambda}}(t,m,d) \mapsto \nu^{\bs{\lambda}(m,d)} \big( \alpha^{\bs{\lambda}(m,d)} (m,d) \big) $  (cf. \eqref{Eq: no entropy reward}).
\end{Theorem}

We see from \eqref{eq: sup-optimal consistent} that when $\kappa = 1$, this asymptotic strategy introduces error of order $\cO(\|d\|)$, compared with $\cO(\|d\|^{1/2})$ for the naive greedy approach (Section \ref{sec:learning premium}). 

The follow theorem shows that our approximate optimal strategy can identify the true model parameter asymptotically when the expected cost is uniformly bounded.
\begin{Theorem}[Complete learning]
\label{Thm: Complete Learning for bounded}
Suppose that (H.\ref{Assump: Instantaneous cost}) and (H.\ref{Assump: Dynamic Asumption}) and let $\bs{\lambda} : \Theta \times \cD \to (0,\infty)$ be such that  $\underline{c} \|d\|^\kappa \leq \bs{\lambda}(m,d) \leq \bar{c}\|d\|^\kappa$ for some constant $\underline{c}, \bar{c} > 0$ and $\kappa \in [0,2]$. Suppose further that
\begin{enumerate}[(i)]
    \item For any policy $\varphi \in \Psi$, the event $\{ \|D^{\varphi,m,d}_t \| \to 0 \} \supseteq \{ \sum_{t=1}^\infty \sI(A^{\varphi,m,d}_t = i) = \infty \; \forall i \in [K] \} $ \footnote{This says that if every arm is chosen infinitely often, then the path $(D^{\varphi,m,d}_t)$ corresponding to $\varphi$ converges to $0$.}.
    \item The cost function $f$ is uniformly bounded.
    \item The function $S(a) := \sup_{u \in \Delta^K} \big(\sum_{i=1}^K u_i a_i + \mathcal{H}(u)\big)$  where $\cH$ is a smooth entropy function (Definition \ref{def: smooth entropy}) satisfies\footnote{The Shannon entropy  $\mathcal{H} (u):= -\sum_{i=1}^K u_i \ln u_i$ satisfies this property.}: for any compact set $K \subseteq \mathbb{R}^K$, there exists a non-empty open ball, $B(r)$, such that $B(r) \cap \del_a S(K) = \emptyset$.
\end{enumerate}
 Then for all $(m,d) \in \Theta \times \cD$, $\|D^{\varphi^{\bs{\lambda}},m,d}_t \| \to 0$ a.s. 
where 
$(D^{\varphi^{\bs{\lambda}},m,d}_t)$  is the path of the $d$ parameter corresponding to the feedback policy $\varphi^{\bs{\lambda}}(m,d) \mapsto \nu^{\bs{\lambda}(m,d)} \big( \alpha^{\bs{\lambda}(t,m,d)} (m,d) \big) $.
\end{Theorem}

\begin{Remark} 
\label{rem:Kalman} 
Our main results assume (H.\ref{Assump: Dynamic Asumption}), which is inspired by the multi-armed bandit problem (Section \ref{sec: Example}) for learning.  The nature of the learning results in a transient state process for $D^\varphi$. This is reflected in (H.\ref{Assump: Dynamic Asumption})$(i)$ where our precision $\|D^{\varphi}_t\|$ is non-increasing over $t \in \sN$. In general, we can extend our approximation result to the case when our precision is recurrent and takes value in a  (small) compact set $\cD$, as often happens when using more general Kalman filtering settings.  All of our analysis follows by using $\|d\| \leq h:=\sup_{d \in \cD} \|d\|$ and $\|d + b_i(m,d)\| \leq h$ for all $i \in [K]$.  In particular, we have the following theorem.
\end{Remark}
\begin{Theorem}
\label{Thm:Kalman}
Suppose that  (H.\ref{Assump: Instantaneous cost}) and (H.\ref{Assump: Kalman}) hold. Then  Theorem \ref{Thm: main bound for value function and approximation} and Theorem \ref{Thm: error bound consistent} hold with all $\|d\|$ in the upper-bound replaced by $h:=\sup_{d \in \cD} \|d\|$. 
\end{Theorem}

\subsection{Summary of the ARC Algorithm}
\label{sec: Summay ARC}
We now summarise how our approximation of the regularised control problem yields an explicit algorithm for a (correlated) multi-armed bandit. 

Recall the setting of our bandit which has $K$ arms with an unknown parameter $\bs{\theta}$ as described at the beginning of Section \ref{sec: Idea ARC}. When the $i$-th arm is chosen, we observe a random variable $Y^{(i)} \sim \pi_i(\cdot | \bs{\theta})$ and obtain a reward $r_i(Y^{(i)})$. Suppose further that the observation distribution $\pi_i$ and the prior $\pi$ of the unknown parameter $\bs{\theta}$ form a conjugate pair for all $i\in [K]$ and we can parameterise the posterior distribution by $(m,d)$ such that the dynamics of these posterior parameters satisfy (H.\ref{Assump: Dynamic Asumption}).   Broadly speaking,  $m$ will be a posterior mean of $\bs{\theta}$ and $d$ will behave as a posterior variance (or some quantity which is inversely proportional to the number of observation).

In summary, to use the ARC algorithm, we need to evaluate some functions explicitly and choose a few hyper-parameters for the decision making.
\paragraph{Hyper-parameters}
\begin{enumerate}[(i)]
\item \textbf{The discount factor $\beta$:} This parameter reflects how long we are considering this sequential decision. A heuristic choice is $\beta = 1-1/T$ where $T$ is the number of rounds of decisions.
\item \textbf{Smooth max approximator $S$:} This is a function to approximate the maximum function (which yields the corresponding entropy $\cH$).  For computational convenience, we propose to take $S(a)=\log\big(\sum_i \exp(a_i)\big)$ which corresponds to the Shannon entropy, $\mathcal{H} (u):= -\sum_{i=1}^K u_i \ln u_i$. More general choices of $S$  can be chosen by considering Definition \ref{def: smooth entropy}. 
\item \textbf{Regularised function $\bs{\lambda}$:} We choose a regulariser (a function of $m$ and $d$) to reflect our learning preferences.  We see in \eqref{eq: sup-optimal consistent} that choosing $\bs{\lambda}(m,d) = \rho \|d\|$ where $\rho > 0$ gives the least order for the sub-optimal bound.  
%However,  when considering the bandit problem with independent arms,  one may take a weighted norm to neglect learning of the bad arms quickly and allow $\bs{\lambda}(m,d)$ to decay to zero faster. 
\end{enumerate}

\paragraph{Relevant functions}
\begin{enumerate}[(i)]
\item \textbf{The expected reward and its first two derivatives:} We need to evaluate $f_i(m,d) := \sE_{m,d}[r_i(Y^{(i)})]$ and compute $\del_d f_i(m,d),  \del_m f_i(m,d)$ and $\del^2_m f_i(m,d)$.
\item \textbf{Evolution of posterior parameter:} Let $(\bs{m}_i, \bs{d}_i)$ be the posterior update after observing  $Y^{(i)}$. We need to evaluate
$$\mu_i(m,d) := \mathbb{E}_{m,d}[\bs{m}_i -m],\q \sigma_i \sigma_i^\top(m,d) := \text{Var}_{m,d}(\bs{m}_i) \q \text{and} \q
 b_i(m,d) := \mathbb{E}_{m,d}[\bs{d}_i -d].$$
 \item \textbf{Learning function:} Given hyper-parameters $(\beta, S, \bs{\lambda})$ and problem environment ($f, \mu_i, \sigma_i, b_i$), we can evaluate the function $(\lambda,m,d) \mapsto L^\lambda(m,d)$ via \eqref{eq: L expression} with $\smaxld(a) := \lambda S(a/\lambda)$ and write $\nu^\lambda(a) := \del_a \smaxld(a) = \del_a S(a/\lambda)$.
\end{enumerate}
  We describe the procedure of the ARC algorithm as follow.

\begin{algorithm}[H]
\label{Alg: ARC}
\SetAlgoLined
%\KwResult{Write here the result }
 \textbf{Input} $m_0,d_0, \beta, S, \bs{\lambda}$ \;
 Set $(m,d) \mapsfrom (m_0, d_0)$ \;
 \For{t = 1, 2, ...}{
Evaluate $U := \nu^{\bs{\lambda}(m,d)}\big(f(m,d) + \beta(1-\beta)^{-1} L^{\bs{\lambda}(m,d)}(m,d) \big)$ \;
 Sample $A \sim \text{Random}( [K],  U)$ \;
 Choose the $A$-th arm, observe $Y^{(A)}$ and collect the reward $R^{(A)}(Y^{(A)})$ \;
 Update posterior parameter $(m,d) $ from the observation $Y^{(A)}$ \;
 }
 \caption{ARC Algorithm }
\end{algorithm}

\section{Comparison with other approaches to bandit problems}
\label{sec: comparison}
\subsection{General Approaches}
\label{sec: bandit review}
 There are various approaches to study bandit problems, and theoretical guarantees are typically proved in specific settings (see e.g. Lattimore and Szepesv\'ari  \cite{Lattimore_book}).  We summarise the broad idea of a few approaches and extend them to our setting using Bayesian inference when needed.  For simplicity, we write $f_i(m,d) := \sE_{m,d}[r_i(Y_t^{(i)})]$ and $\sE_\theta[\cdot] := \sE[\cdot|\bs{\theta} = \theta]$, write $\pi^{\bs{\theta}}_{m,d}$ for the posterior/prior of $\bs{\theta}$ with the parameterisation $(m,d)$  and write $(\bs{m}_i, \bs{d}_i)$ for the posterior update of $(m,d)$ when the $i$-th arm is chosen. 

\begin{itemize}
\item \textbf{$\epsilon$-Greedy ($\epsilon$-GD) \cite{Greedy_regret1, eps_Greedy}:}  At each time,  we choose an arm with the maximal expected reward
$A^{\eps-GD} = \argmax_{i} f_i(m,d)$ with probability $1-\epsilon$; and  choose uniformly at random with probability $\epsilon$.

\item \textbf{Boltzmann Exploration (BE) \cite{BE_Alg, BGE_Alg}:} At each time,  we choose an arm using the probability simplex $U^{BE} = \nu^{\bs{\lambda}(m,d)}\big(f(m,d)\big)$ where $\nu_i^\lambda(a) := \exp(a_i/\lambda)/\big(\sum_j \exp(a_j/\lambda)\big)$.

\item \textbf{Thompson Sampling (TS) \cite{Thompson_original, Tutorial_on_Thompson_Sampling, Russo_Roy_Thompson}:} At each time, a sample $\hat{\bs{\theta}}$ is taken from a posterior distribution $\pi^{\bs{\theta}}_{m,d}$.  We then choose the arm with $A^{TS} = \argmax_{i }\sE_{\hat{\bs{\theta}}}[r_i(Y^{(i)})]$.

\item \textbf{Upper Confidence Bound (UCB) \cite{Original_UCB, UCB_tuned,  Bayes_UCB} :} At each time $t$,  we choose an arm with the maximum index $A^{Bayes-UCB}_t = \argmax_{i} \cQ_{m,d}\big(1-t^{-1} (\log T)^{-c}, \sE_{\bs{\theta}}[r_i(Y^{(i)})]\big)$ where $T$ is the number of plays,  $\cQ_{m,d}(p,X)$ is the $p$-quantile of the random variable $X$ conditional on the posterior parameter $(m,d)$.  Here, $c$ is a hyper-parameter that can be chosen by the decision maker.  Kaufmann et al. \cite{Bayes_UCB} prove a theoretical guarantee of optimal order for the Bernoulli bandit when $c \geq 5$; their simulations suggest that $c=0$ typically performs best.

N.B. There are many variations of the UCB algorithm proposed in various settings. The algorithm described above is known as the Bayes-UCB which has a clear extension to the general setting described in this paper.

\item \textbf{Knowledge Gradient (KG)  \cite{Knowledge_Gradient}:} At each time $t$, we choose an arm with the maximum index
$A^{KG} = \argmax_i \Big(f_i(m,d) + \beta(1-\beta)^{-1} \big( \sE_{m,d}[\max_{j} f_j(\bs{m}_i, \bs{d}_i)] - \max_{j} f_j(m,d)\big)\Big)$.

N.B. In \cite{Knowledge_Gradient}, the KG algorithm is proposed for the classical and linear Gaussian bandit together with an explicit expression for $\sE_{m,d}[\max_{j} f_j(\bs{m}_i, \bs{d}_i)] $.  In general, we may estimate this expression by using Monte-Carlo simulation, but this can be costly.

\item \textbf{Information-Directed Sampling (IDS) \cite{IDS, IDS_HETEROSCEDASTIC_NOISE}:} At each time, for each probability vector $u \in \Delta^K$, we define a (one-step) regret by
$\delta(u) := \sum_{i}u_i \mathbb{E}_{m,d}\big[r_{A^*}(Y^{(A^*)}) - r_i(Y^{(i)})\big]$ 
where $A^* := \argmax_{i}\mathbb{E}_{\bs{\theta}}[r_i(Y^{(i)})]$, and define the information gain by $g(u) := \sum_{i}u_i \big( \cH^{Sh}_{m,d}(\bs{\theta}) - \sE_{m,d}[ \cH^{Sh}_{\bs{m}_i, \bs{d}_i}(\bs{\theta})]\big)$ where $\cH^{Sh}_{m,d}(\bs{\theta})$ is the (differential) entropy of the posterior $\pi^{\bs{\theta}}_{m,d}$.  We then choose an arm using the probability simplex
$U^{IDS} = \argmin_{u \in \Delta^K}\big({\delta(u)^2}/{g(u)}\big)$.

N.B. The information gain $g$ considered above is used in  Kirschner and Krause \cite{IDS_HETEROSCEDASTIC_NOISE} which is different from the original proposed in Russo and  and Roy \cite{IDS}; but is more computationally efficient. 
\end{itemize}
\paragraph{ARC as a combination of the other algorithms:} The ARC algorithm appears as a combination of KG and BE through It\^o's lemma, which results in a random index decision whose index can be decomposed as the sum between Exploitation gain,  $f$, and Learning premium, $L^\lambda$, as in UCB.  This learning premium takes into account asymmetry of available information and a curvature when the reward is non-linear (see section \ref{sec: address limit} for explicit evaluation of $L^\lambda$).  
% This additional exploration gain is closely related to Boltzmann--Gumbel exploration (BGE) proposed by Cesa-Bianchi et al. \cite{BGE_Alg}, where they modify the BE by adding an external noise scaling with uncertainty to encourage learning.  The ARC algorithm is similar to this, in the sense that we randomly make decisions as in the BE algorithm, but a deterministic quantity corresponding to the reduction of uncertainty is added to encourage learning.

More precisely, recall that ARC makes choices based on the (arg)softmax $\nu^\lambda(\alpha(m,d))$ which fundamentally picks an arm with the maximum index $\alpha_i(m,d)= f_i(m,d) + \beta(1-\beta)^{-1} L_i^\lambda(m,d)$. This maximum is determined at random as in the BE, but using a modified index as in UCB.  The decision is modified through the learning term   $L_i^\lambda(m,d)$, which can be seen as 
$L_i^\lambda(m,d)  \approx   \sE_{m,d}[\max_{j} f_j(\bs{m}_i, \bs{d}_i)] - \max_j f_j(m,d)$.  In short, $\alpha(m,d)$ is an approximation of the KG index using smooth max approximation and a second order expansion through It\^o's lemma.  

\paragraph{Computation Efficiency:}
$\epsilon$-GD, BE, TS, UCB and ARC are algorithms where we can often find an explicit expression and thus requires low computational power.  KG, on the other hand, requires evaluation of the expectation of the maximum involving a high-dimension state (which only has an explicit expression in the Gaussian case).  Implementing KG in general can be achieved by Monte-Carlo simulation which is costly.  Similarly,  the IDS requires evaluation of one-step regret and information gain which is expensive in general.  

\subsection{ Shortcomings of bandit algorithms}
\label{sec: limit}
Even though many of the algorithms discussed in Section \ref{sec: bandit review} perform well in many settings,  they may fail to address a few phenomena which may appear in learning.  For clarity, we shall illustrate these shortcomes in extreme scenarios. Many practical examples of these scenarios can be found in Russo and  and Roy \cite{IDS}.

\paragraph{Incomplete Learning:}
Consider a decision rule which depends only on the posterior parameter $(m,d)$. If we start with a bad posterior/prior, we may end up always choosing the worse option, if our early experiences are misleading.

Consider a bandit with $2$ arms: The first arm always gives a fixed reward; the second arm's reward is generated from an unknown distribution. For any strategies that satisfy the property described above, whenever this strategy decides to play the first arm, it will never play the second arm again. However, we can see that if the mean reward of the first arm has unbounded support, the probability that the reward of the first arm is larger than the second arm is strictly positive. This probability never changes when the first arm is abandoned. This means that we have a strictly positive probability of always playing sub-optimal options. 

This is a problem for $\epsilon$-GD (when $\epsilon = 0$) and KG algorithms.
\paragraph{Information Ignorance:}
Many works on bandit assume that all arms have an identical structure. Therefore, they fail to capture the setting where each arm provides different information. 

Consider a bandit with $100$ arms where every arm except the first always gives a strictly positive reward from an unknown distribution.  The first arm is informative but yields no reward; it always gives a reward $0$, but will allow us to observe rewards of all other arms. Playing the first arm helps us to learn $\bs{\theta}$ faster.  Unfortunately, this arm will be ignored by many bandit algorithms as it never has the best reward and many bandit algorithms choose an arm based on the posterior distribution for the reward, but do not consider the information gain.

This is a problem for $\epsilon$-GD (when $\epsilon = 0$), TS,  and UCB algorithms.   It is worth noting that BE and $\epsilon$-GD  (for $\epsilon > 0$) only choose the first arm by random. Hence, they still play worse arms, even if they do not give any meaningful information and never yield the best reward.

\paragraph{Horizon effect:}
A few algorithms are designed based on the principle that the decision should not vary when the terminal time is far away and hence propose a stationary policy. When the horizon is short,  these algorithms may still choose to explore, even if these explorations do not benefit future decisions.

This is a problem for $\epsilon$-GD, BE, TS and IDS algorithms.

\subsection{Addressing Learning using ARC}
\label{sec: address limit}
The ARC algorithm addresses flaws discussed in Section \ref{sec: limit} .  Since we can vary a hyper-parameter $\beta$, which has a natural interpretation as the future weight, we directly address the \textit{horizon effect }.  We also see in Theorem \ref{Thm: Complete Learning for bounded} that ARC is a \textit{complete learning} algorithm in the sense that the uncertainty parameter $(D^\varphi_t)$ 
(posterior variance) 
converges to $0$ as $t \to 0$ which means that we can asymptotically identify the true parameter $\bs{\theta}$. 

To see how ARC address the \textit{information ignorance}, we consider an explicit computation of the ARC algorithm for the Gaussian bandit with additional information (Example \ref{Ex: bandit adding info}).
 
We recall the setting of this example here for convenience of the reader. When the $i$-th arm is chosen, we observe a random variable $\cL(Y^{(i,j)}| \bs{\theta}) \sim N(\bs{\theta}_j, s_{ij}^{-1})$ for $j =1, 2, ..., K$ where $s_{ij} \in [0,\infty)$ and collect the reward $r_i(Y^{(i,i)})$.  We see that the reward of the $i$-th arm depends only on $\bs{\theta}_i$ and the reward of each arm also differs depending on the function $r_i$. Furthermore, when $s_{ij}$ is small,  the variance of the observation $Y^{(i,j)}$ is large. This mean that choosing the $i$-th arm tells us very little about $\bs{\theta}_j$. In other words, our information on the reward of the $j$-th arm does not improve much when $s_{ij}$ is small (and vice versa for large $s_{ij}$). 

Assume that the Shannon entropy (Remark \ref{Rem: Shannon Entropy}) is used as our regulariser.  The decision maker chooses using the (soft-) argmax of the index $\alpha^\lambda(m,d) = f(m,d) + \beta(1-\beta)^{-1} L^\lambda ( m ,d )$ where we can give an explicit expression (see Lemma \ref{lem: add info derive}) for $L^\lambda_i(m,d)$ by
\begin{equation}
\label{eq: L adding info}
L^\lambda_i(m,d) =\frac{1}{2 \lambda} \sum_{j = 1}^K \nu^\lambda_j(f(m,d))(1 - \nu^\lambda_j(f(m,d))) d_{jj}^2\Big(\frac{s_{ij}}{1+ d_{jj}s_{ij}}\Big) \big( \sE_{m,d} [ r'_j(Y^{(j,j)})] \big)^2.
\end{equation}
The term $\beta(1-\beta)^{-1}L^\lambda_i(m,d) $ can be interpreted as a learning premium for choosing the $i$-th arm describing the reduction of uncertainty as discussed in Section \ref{sec:learning premium}.  $\beta(1-\beta)^{-1}$ describes the importance of the future in our preferences.  $L^\lambda_i(m,d) $ comes from summing the (learning) benefit of playing the $i$-th arm to the $j$-th arm:  $d_{jj}^2\big(\frac{s_{ij}}{1+ d_{jj}s_{ij}}\big)$ describes how much we can reduce uncertainty of the $j$-th arm\footnote{One can see from \eqref{eq: Gaussian mean propagation} that this quantity describes the reduction in the (uncertainty) parameter $D$.};  $d_{jj}$ represents the uncertainty, while $s_{ij}$ tells us how much information we would gain.  We also have the term $\big( \sE_{m,d} [ r'_j(Y^{(j,j)})] \big)^2$ which rescales the learning benefit due to how our parameters impact the reward function (in particular, the curvature of the reward).  Finally, the term $\nu^\lambda_j(f(m,d))(1 - \nu^\lambda_j(f(m,d)))$ describes how much we actually need to learn. In particular, when we have an arm $i^*$ such that $f_{i^*}(m,d) \gg f_{j}(m,d)$ for $j \neq i^*$, the $i^*$-th arm is much better than the others, and we probably do not need to learn further.  In this case,  we have $\nu^\lambda(f(m,d)) \approx e_{i^*}$ and $\nu^\lambda_j(f(m,d))(1 - \nu^\lambda_j(f(m,d))) \approx 0$ for all $j$.

% In addition,  if we replace $\lambda$ by a consistent  $\bs{\lambda}$ with order $\kappa = 1$, we see that $L^\lambda_i(m,d) \sim \|d\|$ (which is slightly different from the Learning Premium of the UCB, which scales with $\|d\|^{1/2}$).  Roughly speaking,  $d$ can also be interpreted as $(1/n_1, ...,1/n_K)$ where $n_i$ is the number of attempts on the $i$-th arm.  When playing an arm, our parameter estimate will naturally change by order $1/n$ as we can see in e.g. \eqref{eq: Gaussian mean propagation}, \eqref{eq: Binomial dynamic} or \eqref{eq: Poisson dynamic}.  

\section{Numerical Experiments}
\label{sec: Numerical}
In this section, we run numerical experiments to illustrate the accuracy of the approximation and run a simulation to compare performance of the ARC algorithm to other algorithms. 
\subsection{Comparison to the optimal solution for $1\tfrac{1}{2}$ bandit}
\label{sec: compare to the optimal}
To illustrate that our approximation gives a reasonable answer, we compare our estimated value function and its corresponding control to the exact value function in a simple setting.

Suppose that our bandit has two arms. The first arm always gives the reward $Y \sim N(\theta,1)$; the second arm always gives reward $1$.  We observe the reward of the first arm only when the first arm is chosen.  Formulating this as a relaxed control \eqref{eq: infinite horizon value} with dynamic in \eqref{eq: Gaussian mean propagation} gives
$$V^\lambda_\infty(m,d) = \sup_{\varphi \in \Psi} \sE_{m,d} \bigg[ \sum_{t=0}^\infty \beta^t \Big( \big(U^\varphi_{1,t+1}M^\varphi_t + U^\varphi_{2,t+1} \big) + \lambda \cH(U^\varphi_{t+1}) \Big)\bigg] \q ; \q \cH(u) = -\sum_{i=1}^2 u_i \ln u_i$$
with $\Theta = \sR$ and $\cD = [0, 1]$. The transition \eqref{eq: transition map} of the problem is given by
\begin{equation*}
\Phi(m,d, i,\xi) := 
\begin{cases}
(m,d) + \big(
d(1+d)^{-1/2} z(\xi),  - d^2(1+d)^{-1}
\big) & ; \; i =1 \\
(m,d) & ; \; i =2
\end{cases} \q \text{where} \;z(\bs{\xi}_t) \sim_{IID} N(0,1).
\end{equation*}
We solve the above problem explicitly using Monte Carlo simulation. In particular, we start our iteration with $\tilde{V}^\lambda_0(m,d) = 0$ and iteratively compute on the grid $(m,d)$,
$$\tilde{V}^\lambda_{n+1}(m,d) = \sup_{u \in \Delta^2} \Bigg\{u_1\Big(m + \frac{\beta}{N}\sum_{i=1}^N\tilde{V}^\lambda_{n} \big( \Phi(m,d, 1,\bs{\xi}_{i,n}) \big)\Big)  + u_2 \Big( 1 + \beta \tilde{V}^\lambda_{n}(m,d) \Big) + \lambda \cH(u) \Bigg\}$$
where $N$ is the number of Monte Carlo simulations and $(\bs{\xi}_{i,n})$ are such that $z(\bs{\xi}_{i,n}) \sim_{IID} N(0,1)$. We then interpolate $\tilde{V}^\lambda_{n+1}$ and repeat the procedure until $\tilde{V}^\lambda_n$ converges to  $\tilde{V}^\lambda$.  The corresponding optimal (feedback) probability to choose the first arm is given by
$$\tilde{p}^{\lambda,*}(m,d) \approx\frac{\exp\Big( \tfrac{1}{\lambda}\Big(m + \beta\big( \frac{1}{N}\sum_{i=1}^N\tilde{V}^\lambda( \Phi(m,d, 1,\bs{\xi}_{i}) \big) - \tilde{V}^{\bs{\lambda}}(m,d) \big) \Big)\Big)}{\exp\Big( \tfrac{1}{\lambda}\Big(m + \beta\big( \frac{1}{N}\sum_{i=1}^N\tilde{V}^\lambda( \Phi(m,d, 1,\bs{\xi}_{i}) \big) - \tilde{V}^{\bs{\lambda}}(m,d) \big) \Big) + \exp\big(\tfrac{1}{\lambda} \big)},   $$
where $z(\bs{\xi}_{i}) \sim_{IID} N(0,1)$. On the other hand, our ARC approximation gives 
{\small
$$V^{\lambda, ARC}_\infty(m,d) = (1-\beta)^{-1} \lambda \log \Big( \exp\Big(\frac{ \alpha_1^{{\lambda}}(m,d)}{\lambda} \Big) + \exp\Big(\frac{1}{\lambda} \Big)\Big), \q  p^{\lambda,ARC}(m,d) = \frac{\exp\big(\frac{ \alpha_1^{{\lambda}}(m,d)}{\lambda} \big)}{\exp\big(\frac{ \alpha_1^{{\lambda}}(m,d)}{\lambda} \big) +  \exp\big(\frac{1}{\lambda} \big)}, $$
}
where $\alpha_1^{\lambda}(m,d) = m + \tfrac{1}{2\lambda} \beta(1-\beta)^{-1} \nu_1^\lambda(m)(1-\nu_1^\lambda(m)) d^2(1+d)^{-1}$ and $\nu_1^\lambda(m) = \frac{\exp(m/\lambda)}{\exp(m/\lambda) + \exp(1/\lambda)}$.

We now compare $\tilde{V}^\lambda, \tilde{p}^{\lambda,*}$ with $V^{\lambda, ARC}_\infty, p^{\lambda,ARC}$ when $\lambda = 0.1$ and $\beta = 0.99$. In the numerical experiment, we use $N = 1000$ and consider $m \in [0,2]$ and $d \in [1/100, 1/20]$. We see in Figure \ref{fig:value_plot} (which shows the value functions) and Figure \ref{fig:prob_plot} (which shows the probability of choosing the risky arm) that the ARC approximation gives a close estimate of the regularised problem.

\begin{figure}[h]
	\centering
	\includegraphics[clip, trim=0cm 0.1cm 4.5cm 0.5cm, width=0.8\linewidth]{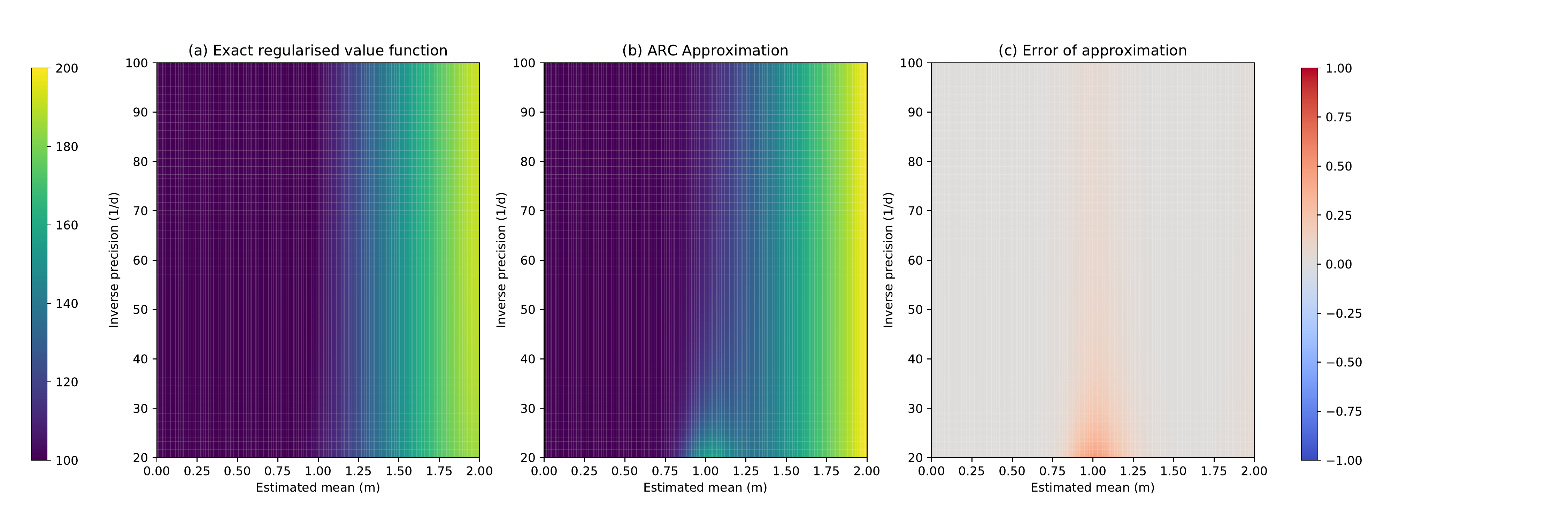}
	\caption{\small (a) $\tilde{V}^\lambda \qq$ (b) $V^{\lambda, ARC}_\infty \qq$   (c) $(V^{\lambda, ARC}_\infty - \tilde{V}^\lambda)/\tilde{V}^\lambda$}
	\label{fig:value_plot}
\end{figure}
\begin{figure}[h]
	\centering
	\includegraphics[clip, trim=0cm 0.1cm 4.5cm 0.5cm, width=0.8\linewidth]{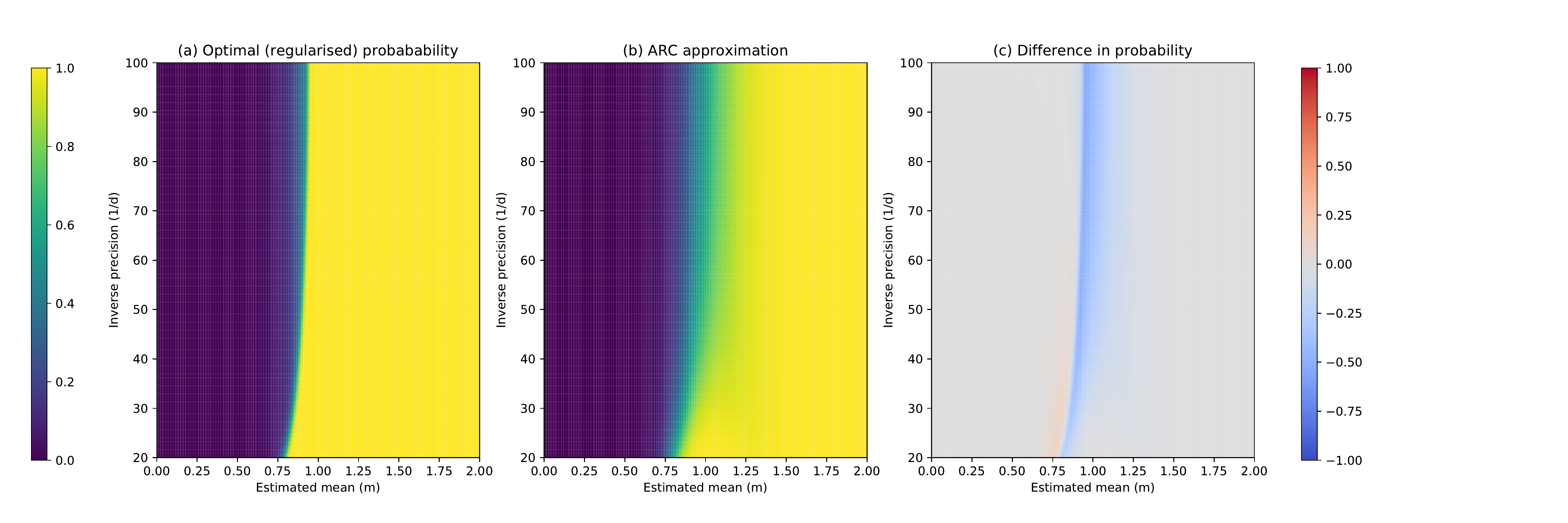}
	\caption{(a) $\tilde{p}^{\lambda,*} \qq$ (b) $p^{\lambda, ARC}_\infty \qq$   (c) $p^{\lambda, ARC}_\infty - \tilde{p}^{\lambda,*}$}
	\label{fig:prob_plot}
\end{figure}
\subsection{Simulation results}
A common performance measure which is used in the multi-armed bandit literature \cite{IDS, GLM_bandit, adaptive_learning_survey, Bayes_UCB, IDS_HETEROSCEDASTIC_NOISE} is the regret
$$R(A, T, \bs{\theta}) = \sum_{t=1}^{T}\Big( \max_{i \in [K]} \mathbb{E}\big[R^{(i)}(Y^{(i)}) \big| \bs \theta \big] - \mathbb{E}\big[R^{(A_t)}(Y^{(A_t)}) \big| \bs \theta \big]\Big).$$

In this section, we will compare the regret of the ARC algorithm with other approaches described in Section \ref{sec: bandit review} under $3$ environments; classical bandit, bandit with an informative arm and linear bandit.  In each of these environments, we consider decisions on a $50$-armed bandit with horizon $T = 2 \times 10^3$ steps over $10^3$ simulations.  In the $n$-th simulation, we sample the true parameter $\bs{\theta}^{(n)} \sim N(\bm{1}_{50},  I_{50})$ and then simulate interaction between each of the algorithms and the environment with the  parameter $\bs{\theta}^{(n)}$ starting from an initial belief $(m_0, d_0) = (0,10^3 \times I_{50})$ corresponding to a non-informative prior.  We then use our simulated output to demonstrate means, medians, 0.75 quantiles and 0.90 quantiles of the map $t \mapsto  R(A, t, \bs{\theta})$ for each of our considered algorithms.

\paragraph{Simulation Environment:} We consider the following environment for our simulation. Let $\bs{\theta}$ be an unknown parameter taking values in $\sR^{50}$. 
\begin{itemize}
\item \textit{Classical bandit.} When choosing the $i$-th arm, we observe and receive the reward sampled from the distribution $N(\bs{\theta}_i,  5)$.
\item \textit{Bandit with an informative arm.}  When choosing the $i$-th arm with $i \neq 1$, we observe and receive the reward sampled from the distribution $N(\bs{\theta}_i,  5)$. When the $1$-st arm is chosen,  we receive a reward sampled from $N(\bs{\theta}_1-1,  5)$ and in addition, we observe a sample from $N(\bs{\theta},  5 \times I_{50})$.  In particular, playing the first arm allows us to observe the rewards of other arms without choosing them, but this arm yields a smaller reward than others.
\item \textit{Linear bandit.} When choosing the $i$-th arm, we observe and receive a reward sampled from the distribution $N(b_i^\top \bs{\theta},  5)$ where $b_i = e_i + e_{i+1}$ for $i \neq 50$ and $b_{50} = e_1 + e_{50}$.
\end{itemize}

\paragraph{Hyper-parameters of bandit algorithms:} We will consider the KG and IDS by introducing $100$ Monte-Carlo samples to evaluate the required expectations.  We will set the parameter $\beta = 0.9995$ for KG and ARC.  The function $\bs{\lambda}(m,d)$ considered in the BE and ARC will be given in the form $\bs{\lambda}(m,d) = \rho \| d \|_{op}$ where $\|\cdot\|_{op}$ is the matrix operator norm and use Shannon entropy (Remark \ref{Rem: Shannon Entropy}) to regularise the ARC algorithm.  Unmentioned hyper-parameters will appear as a description of the algorithm in the regret plot. 

\paragraph{ARC index strategy:} For numerical efficiency, we introduce an index strategy inspired by the ARC algorithm. In contrast to the ARC, instead of making a decision based on the probability simplex $\nu^{\bs{\lambda}(m,d)}\big(\alpha^{\bs{\lambda}(m,d)}(m,d)\big)$ with $\alpha^{\lambda}(m,d)$ given in \eqref{eq: alpha},  we simply choose the arm $i$ which maximises $\alpha_i^{\bs{\lambda}}(m,d)$. 

\begin{figure}[h]
	\centering
	\includegraphics[clip, trim=7cm 0.5cm 7cm 0.5cm, width=1\linewidth]{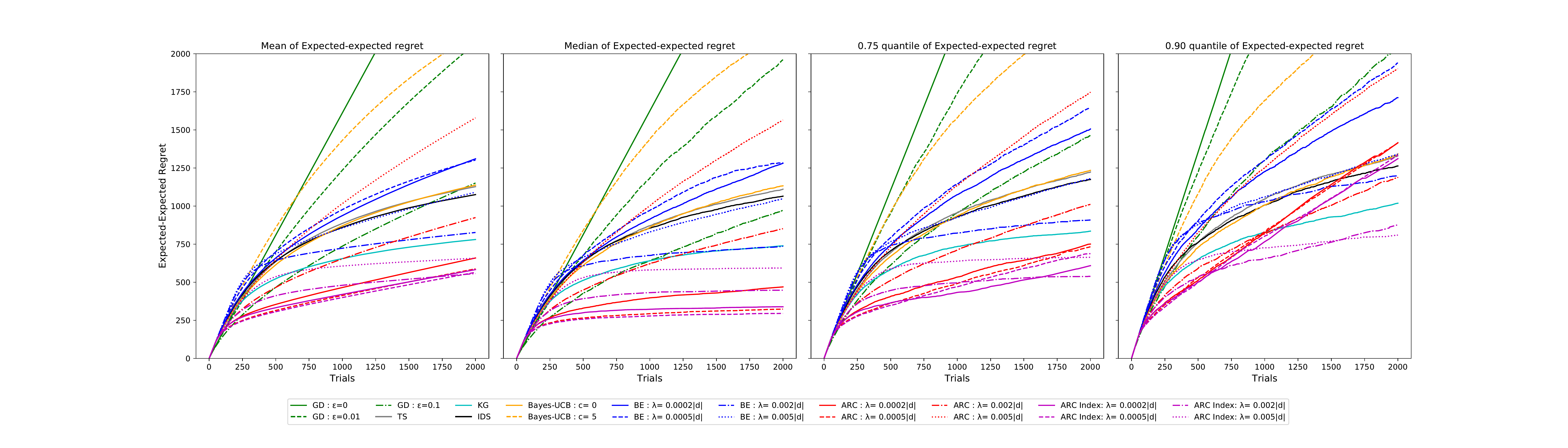}
	\caption{Regret for the classical bandit}
	\label{fig:classic}
\end{figure}

\begin{figure}[h]
	\centering
	\includegraphics[clip, trim=7cm 0.5cm 7cm 0.5cm, width=1\linewidth]{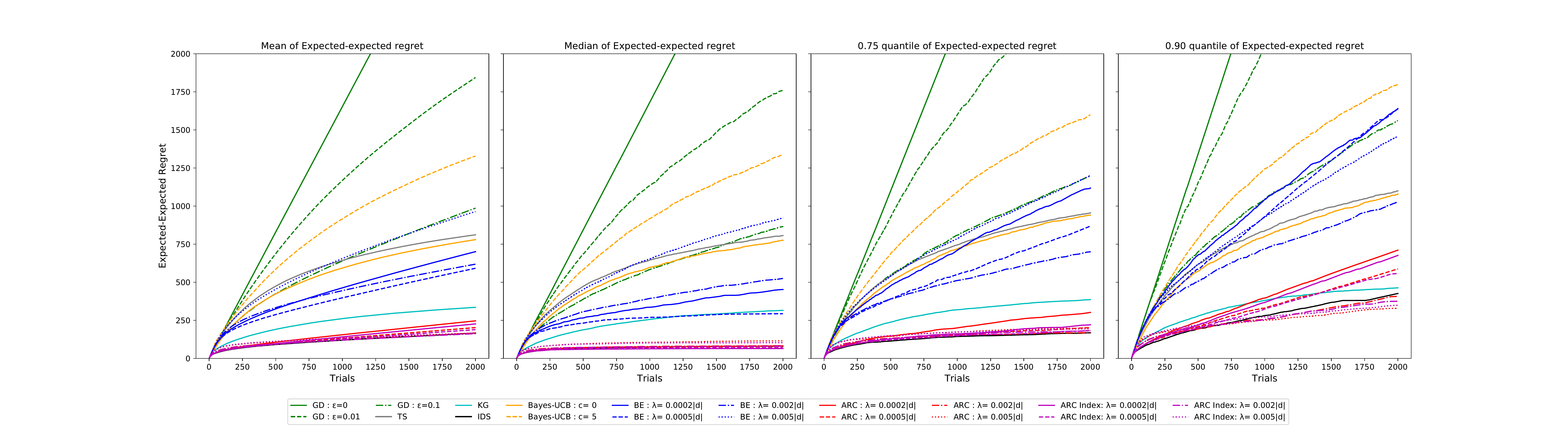}
	\caption{Regret for the bandit with an informative arm}
	\label{fig:info}
\end{figure}

\begin{figure}[h]
	\centering
	\includegraphics[clip, trim=7cm 0.5cm 7cm 0.5cm, width=1\linewidth]{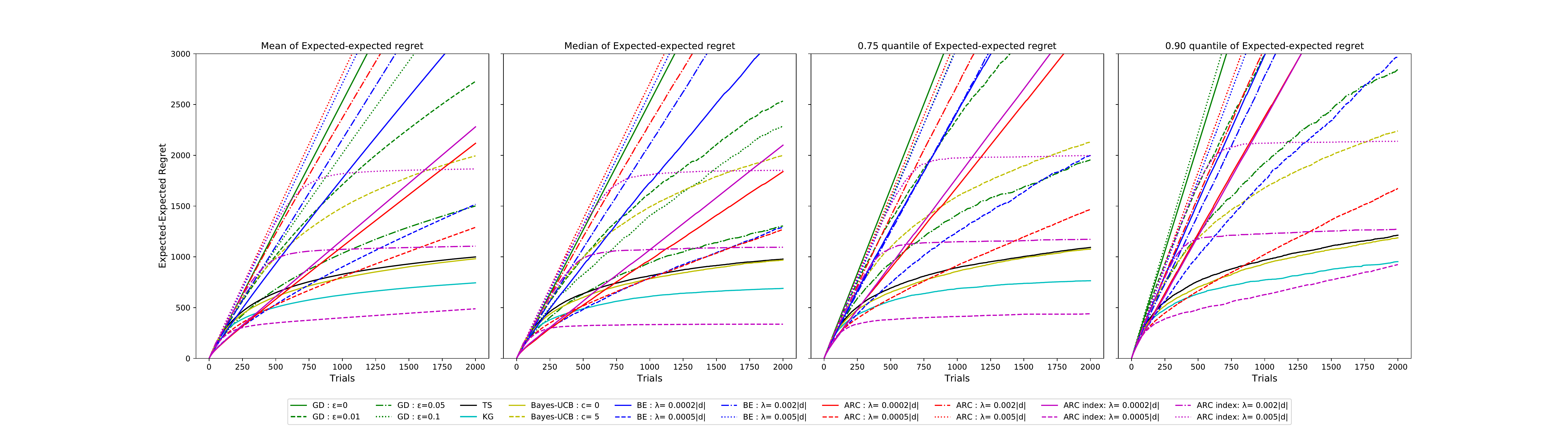}
	\caption{Regret for the linear bandit (IDS excluded due to computational cost)}
	\label{fig:linear}
\end{figure}
\paragraph{Discussion of numerical simulations:} We see in Figures \ref{fig:classic},  \ref{fig:info} and  \ref{fig:linear} that both ARC and ARC index strategy with appropriate hyper-parameter perform very well compared to other algorithms.  We see that the ARC approaches perform significantly better than other approaches in the setting for the bandit with an informative arm.  This performance is as good as IDS but requires significantly lower computational cost, since every term can be evaluated explicitly.  In fact, the computational cost of the IDS proved too expensive to demonstrate for the linear bandit with $50$ arms and was therefore omitted.

The ARC algorithm derived in this paper performs well with appropriate hyper-parameter choices.  It is also worth pointing out one  obstacle found in the derived ARC algorithm when running numerical simulations with many arms available.  We see that taking $\bs{\lambda}(m,d) = \rho \| d \|$ may not allow $\|d\|$ to decay sufficiently fast (even though Theorem \ref{Thm: Complete Learning for bounded} ensures that this ultimately converges to $0$). In particular,  with many arms,  we should be able to identify some arms which are significantly worse than the others. Such arms will be rarely played, which leaves $\|d\|$ large for a long time.  We found in the numerical simulation that when following the ARC procedure, the agent identifies the reward of the very best few arms correctly (i.e. we obtain close estimates to the rewards of a few best arms).  Unfortunately, since $\|d\|$ does not decay sufficiently fast, this leads the ARC algorithm to randomly choose among some small number of options, even if it identifies the best arm correctly.  This explains why we observe linear trends in the regret plot for the simple ARC but with a low gradient. 

To overcome this effect, we considered the ARC index strategy, which does not require $\|d\|$ to decay to zero to terminate decision to a single (best) arm.  Here, we see in Figure \ref{fig:classic},  \ref{fig:info} and  \ref{fig:linear} that  the gradient of the regret of the ARC index converges to zero, that is, we eventually identify the best arm.  In general, one may also choose the function $\bs{\lambda}(m,d)$ depending on $m$ to truncate our consideration to the best arm. 
% or one may also introduce a concatenation of functions to allow the ARC to neglect the size of $\|d\|$ at later stages as in Cesa-Bianchi et al. \cite[Theorem 2]{BGE_Alg}. 

\section{Proof of the main results}
\label{sec: Prove ARC}
% We now flesh out our heuristic derivation in Section \ref{sec: core derivation ARC}.
\subsection{Convex analysis and smooth max approximator}
In \eqref{eq: smax} and \eqref{eq: nu}, we briefly introduce $\smaxld$ and $\nu^\lambda$ as a smooth version of the maximum function and its derivative, obtained via the convex conjugation of the regularisation function $\cH$.  One well-known choice of $\cH$ which gives an analytical expression is the Shannon entropy. We can also consider other choices of $\mathcal{H}$ by constructing the smooth max approximator  $\smaxld$ explicitly\footnote{Examples of explicit constructions can be found in  Zhang and Reisinger \cite[Remark 3.1]{John_relaxed_control}.}. 

We observe that $\smaxld$ in \eqref{eq: smax} can be expressed as $\smaxld (a) = \lambda S(a/\lambda)$ where
\begin{equation}
\label{eq: robust rep}
S(a) = \sup_{u \in \Delta^K} \Big(\sum_{i=1}^K u_i a_i + \mathcal{H}(u)\Big).
\end{equation}
In particular, $-\mathcal{H}$ is the convex conjugate of $S$ (see e.g. Rockafellar \cite{Convex_Analysis_book}).  In fact, \eqref{eq: robust rep} is also known as a `nonlinear expectation' \footnote{Nonlinear Expectations (or equivalently  `risk measures') are a classical tool in mathematical finance to study decision making under uncertainty.} defined on a finite space (see Coquet et al. \cite{represent_nonlinear}).
% and \eqref{eq: robust rep} gives the `robust representation' for nonlinear expectation.

\begin{Definition}
	We say a function $S: \mathbb{R}^K \to \mathbb{R}$ is a \textit{convex nonlinear expectation} if it satisfies:
	\begin{enumerate}[(i)]
		\item \textbf{Monotonicity}: If $a \leq b$, then $S(a) \leq S(b)$;
		\item \textbf{Translation Equivariance}: For all $c \in \mathbb{R}$, $S(a + c \bm{1}_K) = S(a) + c$;
		\item \textbf{Convexity}: For any $\kappa \in [0,1]$, $S(\kappa a + (1-\kappa) b) \leq \kappa S(a) + (1-\kappa) S(b)$;
	\end{enumerate}
	where the inequalities are interpreted component-wise.
\end{Definition}

The following theorem shows that using $\smaxld$ as a smooth max approximator is equivalent to having $\cH$ bounded.  This will allow us to quantify the difference between a non-regularised \eqref{Eq: no entropy reward} and regularised control problem \eqref{eq: infinite horizon value} by $\cO (\lambda)$. The proof is given in Appendix \ref{sec:Proofs of relevant results}.
\begin{Theorem}
	\label{thm: equivalence to boundedness}
	Let $S$ be a convex nonlinear expectation. The following are equivalent.
	\begin{enumerate}[(i)]
		\item There exists $N \in \mathbb{R}$ such that $S(a) +N \geq \max_i a_i$ for all $a \in \mathbb{R}^K$.
		\item There exists $N \in \mathbb{R}$ such that $\mathcal{A}_S :=\{a  \in \mathbb{R}^K: S(a) \leq 0\} \subseteq (-\infty,N]^K$.
		\item There exists a bounded function $\mathcal{H}: \Delta^K \to \mathbb{R}$ such that \eqref{eq: robust rep} holds.
		\item For $\smaxld (a) = \lambda S(a/\lambda)$, we have $\sup_{a \in \mathbb{R}} |\smaxld(a) - \max_i a_i| \to 0$ as $\lambda \downarrow 0$.
	\end{enumerate}
\end{Theorem}

We now introduce a smooth entropy regulariser as the convex conjugate of a smooth max approximator.
 \begin{Definition}
	\label{def: smooth entropy}
	We say a function $S: \mathbb{R}^K \to \mathbb{R}$ is a \textit{smooth max approximator} if it is a $3$-times differentiable convex nonlinear expectation with uniformly bounded derivatives such that Theorem \ref{thm: equivalence to boundedness} holds. We say a bounded function $\mathcal{H} : \Delta^K \to \mathbb{R}$ is a \textit{smooth entropy} if $-\mathcal{H}$ is a convex conjugate of some smooth max approximator\footnote{One can check that the Shannon Entropy (Remark \ref{Rem: Shannon Entropy}) is a smooth entropy.}.
	
	For a smooth max approximator $S$, we write $\smaxld(a) := \lambda S(a/\lambda)$,  $ \nu^\lambda(a):= \del_y S \big|_{y = a/\lambda}  = \del_a \smaxld(a)$ and $\eta^\lambda(a) := \del^2_y S\big|_{y = a/\lambda} = \lambda \del^2_a \smaxld(a)$.
\end{Definition}

\begin{Remark}
\label{nu and eta remark}
If $S$ is a smooth max approximator, then $\nu^\lambda$ and $\eta^\lambda$ are uniformly bounded. Moreover, it follows from Fenchel's inequality that 
	$\nu^\lambda(a) = \argmax_{u \in \Delta^K}  \big(\sum_{i=1}^K u_i a_i + \lambda\mathcal{H}(u)\big).$
	In particular,  $\nu^\lambda$ can be interpreted as a smooth version of the argmax, as in \eqref{eq: nu}.
\end{Remark}

\subsection{Analysis of the regularised control problem over finite horizon }

The objective of this section is to approximate the finite horizon value function as a sum of the (smooth) maximum of the incremental rewards.  

\begin{Definition}
\label{def: finite horizon value}
Let $\mathcal{H}: \Delta^K \to \mathbb{R}$ be a smooth entropy (Definition \ref{def: smooth entropy})  and $\lambda > 0$. Define
\begin{equation}
	\label{eq: finite horizon value}
	\left.
	\begin{aligned}
	V^\lambda_T(m,d) &:= \sup_{\varphi \in \Psi} V^{\lambda,\varphi}_T(m,d), \qquad \text{where} \\
	V^{\lambda,\varphi}_T(m,d) &:= \mathbb{E}_{m,d}\Bigg[\sum_{t=0}^{T-1}\beta^{t} \bigg( \Big(\sum_{i=1}^K f_{i}\big(M^{\varphi}_t, D^{\varphi}_t \big) U^\varphi_{i, t+1}\Big) + \lambda \mathcal{H}(U^\varphi_{t+1}) \bigg)\Bigg].
\end{aligned}	
\right\}
	\end{equation}
	\end{Definition}
	
We will show that
	$
 V^\lambda_{T}(m,d)  \approx \sum_{t=0}^{T-1} \beta^{t}\big(\smaxld \circ \alpha_{T-t}^\lambda\big)(m,d)$, with an approximate optimal policy $\varphi^{\lambda,  T}(t,m,d) := \nu^\lambda \big( \alpha^\lambda_{T-t+1}(m,d) \big)$. Here, the incremental reward with $t$-steps to go is
\begin{equation}
    \label{eq: alpha_t}
    \alpha_t^\lambda(m,d) := f(m,d) + L^\lambda(m,d) \bigg( \sum_{s=1}^{t-1}\beta^s \bigg),
\end{equation}
where $L^\lambda(m,d)$ is given in \eqref{eq: L expression}.

The idea behind the analysis is to consider an asymptotic expansion as $\|d\| \to 0$.  Due to our learning structure (H.\ref{Assump: Dynamic Asumption}), the change in the underlying state $(m,d)$ are (in expectation) of order $\cO(\|d\| ^2)$.  Hence, the global Lipschitz property of $f$  implies that the instantaneous reward changes with $\cO(\|d\| ^2)$.  Hence, we can use Taylor's theorem to obtain an asymptotic expansion in $\|d\| $ and keep the terms up to order $\cO(\|d\| ^2)$. 
% Throughout the proof, we shall quantify the error bounds in terms of
% \begin{equation}
%     \label{eq:poly bound}
%     P_{\lambda, d} (m,n) :=   \big(1+ \lambda^{-m} \big) \| d \|^n \q ;  \q \lambda
% \in (0,\infty),  \q d \in \cD,  \q m,n \in \sN.
% \end{equation}

We first show that the perturbation error in the learning term $L^\lambda$ is of order $\cO(\|d\| ^3)$ and can be ignored in our approximation. The proofs of the following two lemmas are in Appendix \ref{sec:Proofs of relevant results}.
\begin{Lemma}
\label{lem: error of L}
Suppose that (H.\ref{Assump: Instantaneous cost}) and (H.\ref{Assump: Dynamic Asumption}) hold, and  let $S$ be a smooth max approximator (Definition \ref{def: smooth entropy}).  There exists a constant $C > 0$  such that the function $(\lambda,m,d) \mapsto L^\lambda(m,d)$ given in \eqref{eq: L expression} satisfies: for any $\lambda > 0$, $i \in [K]$,  $(m,d) \in \Theta \times \cD$ and $t \in \sN$, 
$$\sE \big|L^\lambda \big( \Phi(m,d, i, \bs{\xi}_t) \big) - L^\lambda(m,d) \big|\leq C \|d\|^3 (1+ \lambda^{-2} ).$$

\end{Lemma}
Now, we consider the second order approximation of the smooth maximum $\smaxld$ over the incremental reward $\alpha$.  We show that $L^\lambda$ is the second order approximation (in expectation) of the (smooth) maximum incremental reward.

\begin{Lemma}
\label{lem: one step return error}
Suppose that (H.\ref{Assump: Instantaneous cost}) and (H.\ref{Assump: Dynamic Asumption}) hold , and  let $S$ be a smooth max approximator (Definition \ref{def: smooth entropy}) and $\smaxld(a) := \lambda S(a/\lambda)$.  There exists a constant $C > 0$  such that for any $\lambda > 0$, $i \in [K]$,  $(m,d) \in \Theta \times \cD$ and $t,T \in \sN$,
\begin{align*}
&\Big| \sE  \big(\smaxld\circ \alpha_T^\lambda\big) \big( \Phi(m,d,  i, \bs{\xi}_t) \big) - \big(\smaxld\circ \alpha_T^\lambda\big)(m,d)- L_i^\lambda (m,d)\Big|   \leq C \|d\|^3\big(1 + \lambda^{-2} + \lambda^{-3} \|d\|  \big)
\end{align*}
%In addition, if (H.\ref{Assump: Linear}) also holds, then the upper-bound  becomes
%$ C\big(P_{\beta, \lambda, d} (0, 2,3)+ P_{\beta, \lambda, d} (1, 1,3) + P_{\beta, \lambda, d} (1, 3,4)  \big).$
where $\alpha^\lambda_t$ is defined in \eqref{eq: alpha_t}, $L^\lambda$ is defined in \eqref{eq: L expression}.
\end{Lemma}
Since $\alpha^\lambda_{t}(m,d)$ describes an incremental reward with $t$-steps to go, we may approximate the optimal value function \eqref{eq: finite horizon value} by $\sum_{t=0}^{T-1} \beta^{t}\big(\smaxld \circ \alpha_{T-t}^\lambda\big)(m,d) $ with an estimate
optimal strategy $\varphi^{\lambda,  T}(t,m,d) := \nu^\lambda \big( \alpha^\lambda_{T-t}(m,d) \big)$.

\begin{Theorem}
\label{Thm: Expression of V_T}
Suppose that (H.\ref{Assump: Instantaneous cost}) and (H.\ref{Assump: Dynamic Asumption}) hold, and  let $S$ be a smooth max approximator (Definition \ref{def: smooth entropy}) and $\smaxld(a) := \lambda S(a/\lambda)$.  There exists a constant $C > 0$ such that the value function for finite horizon (Definition \ref{def: finite horizon value}) satisfies: for any $\lambda > 0$,  $(m,d) \in \Theta \times \cD$ and $T \in \sN$,
\begin{align}
\label{eq: ARC finite horizon value}
\Big| V^{\lambda, \varphi^{\lambda,  T}}_{T}(m,d)  - \sum_{t=0}^{T-1} \beta^{t}\big(\smaxld \circ \alpha_{T-t}^\lambda\big)(m,d) \Big| &\leq C \|d\|^3\big(1 + \lambda^{-2} + \lambda^{-3} \|d\|  \big), \q \text{and}  \\
\label{eq: finite horizon value bound}
\Big| V^\lambda_{T}(m,d)  - \sum_{t=0}^{T-1} \beta^{t}\big(\smaxld \circ \alpha_{T-t}^\lambda\big)(m,d) \Big| &\leq C \|d\|^3\big(1 + \lambda^{-2} + \lambda^{-3} \|d\|  \big).
\end{align}
where $\varphi^{\lambda,  T}(t,m,d) := \nu^\lambda \big( \alpha^\lambda_{T-t+1}(m,d) \big) = \del_a \smaxld \big( \alpha^\lambda_{T-t+1}(m,d) \big)$ and $\alpha^\lambda_t$ is defined in \eqref{eq: alpha_t}.
\end{Theorem}
\begin{proof} 
We will prove by induction that the upper-bounds for  \eqref{eq: ARC finite horizon value} and \eqref{eq: finite horizon value bound} are given by $Q_{\lambda,d} = (1-\beta)^{-2}  C \|d\|^3\big(1 + \lambda^{-2} + \lambda^{-3} \|d\|  \big)$ where $C \geq 0$ is a constant in Lemma \ref{lem: one step return error}.

It is clear from \eqref{eq: smax} that $V^\lambda_1 = \smaxld \circ \alpha_{1}^\lambda$. Hence, the required inequality  holds for $T =1$. 

Assume that the required inequality holds for $T-1$. 

Define $ R^i_T(m,d) := \beta \mathbb{E}\big[V^{\lambda, \varphi^{\lambda,  T-1}}_{T-1} \big( \Phi(m,d, i, \bs{\xi}_1) \big) \big] - \sum_{t=1}^{T-1} \beta^{t}\left( \big(\smaxld\circ \alpha_t^\lambda\big)(m,d) +  L_i^\lambda (m,d) \right).$ We see that
\begin{align}
\label{eq:derived values action}
&V_{T}^{\lambda,\varphi^{\lambda,T}}(m,d) = \sum_{i=1}^KU^{\varphi^{\lambda,T},m,d}_{i,1}\Big(f_i(m,d) + \beta \mathbb{E}\big[V_{T-1}^{\lambda,\varphi^{\lambda,T-1}}\big( \Phi(m,d, i, \bs{\xi}_1)\big) \big]\Big) + \lambda \mathcal{H}(U^{\varphi^{\lambda,T},m,d}_{1}) \nonumber \\
&\quad =\bigg(\sum_{i=1}^KU^{\varphi^{\lambda,T},m,d}_{i,1}\Big(\alpha^\lambda_{T,i}(m,d) + R^i_T(m,d) \Big)+ \lambda \mathcal{H}(U^{\varphi^{\lambda,T},m,d}_{1})\bigg) +   \sum_{t=1}^{T-1} \beta^{t} \big(\smaxld \circ \alpha^\lambda_{T-1-t}\big)(m,d) \nonumber \\
&\quad =\bigg(\sum_{i=1}^KU^{\varphi^{\lambda,T},m,d}_{i,1} R^i_T(m,d) \bigg) +  \big(\smaxld \circ \alpha^\lambda_{T}\big)(m,d) +   \sum_{t=1}^{T-1} \beta^{t} \big(\smaxld \circ \alpha^\lambda_{T-1-t}\big)(m,d) 
\end{align}
 where the second equality follows from substituting $\alpha^\lambda_{T,i}(m,d) = f_i(m,d) + \beta \big( \sum_{t=1}^{T-1} \beta^{t}  \big) L_i^\lambda(m,d) $ and the final inequality follows from \eqref{eq: smax}-\eqref{eq: nu} and the fact that  $$U^{\varphi^{\lambda,T},m,d}_{1} = \del_a \smaxld \big( \alpha^\lambda_{T}(m,d) \big) = \argmax_{u \in \Delta^K} \bigg(\sum_{i=1}^K u_i \alpha^\lambda_{T,i}(m,d) + \lambda \cH(u) \bigg).$$
 By our inductive hypothesis, Lemma \ref{lem: one step return error} and (H.\ref{Assump: Dynamic Asumption})$(i)$, we see that
\begin{align*}
    |R^i_T(m,d)| &:= \beta \mathbb{E}\left|V^{\lambda, \varphi^{\lambda,  T-1}}_{T-1} \big( \Phi(m,d, i, \bs{\xi}_1) \big) -\sum_{t=0}^{T-2} \beta^{t}\big(\smaxld \circ \alpha_{T-t}^\lambda \big)\big( \Phi(m,d, i, \bs{\xi}_1) \big) \right| \\
    &\qquad + \sum_{t=0}^{T-2} \beta^{t}\left|\sE  \big[\big(\smaxld\circ \alpha_t^\lambda\big) \big( \Phi(m,d, i, \bs{\xi}_1) \big) \big] - \big(\smaxld\circ \alpha_t^\lambda\big)(m,d) - L_i^\lambda (m,d) \right| \\
    & \leq \beta Q_{\lambda, d+b_i(m,d)} + \Big(\sum_{t=0}^{T-2} \beta^{t} \Big) (1-\beta)^2 Q_{\lambda,d} \leq \beta Q_{\lambda, d} + (1-\beta)Q_{\lambda, d} = Q_{\lambda, d}.
\end{align*}
Therefore, $\big| \sum_{i=1}^K U^{\varphi^{\lambda,T},m,d}_{i,1} R^i_T(m,d) \big| \leq Q_{\lambda, d}$ and the inequality \eqref{eq: ARC finite horizon value} for $T$ follows from \eqref{eq:derived values action}.

Similarly, to prove \eqref{eq: finite horizon value bound},  we define $\tilde{R}^i_T(m,d)$ in the same manner as  ${R}^i_T(m,d)$ but replace $V^{\lambda, \varphi^{\lambda,  T-1}}_{T-1}$  by $V^{\lambda}_{T-1}$.  The similar argument yields $|\tilde{R}^i_T(m,d)| \leq Q_{\lambda,d}$ for all $T \in \sN$ and $i \in [K]$. 

It follows from the dynamic programming principle that
\begin{align*}
&V_{T}^\lambda(m,d) = \sup_{u \in \Delta^K}\bigg\{\sum_{i=1}^Ku_{i}\Big(f_i(m,d) + \beta \mathbb{E}\big[V^\lambda_{T-1}\big( \Phi(m,d, i, \bs{\xi}_1)\big) \big]\Big) + \lambda \mathcal{H}(u)\bigg\}
\\
&\quad = \sup_{u \in \Delta^K}\bigg\{\bigg( \sum_{i=1}^Ku_{i} \alpha^\lambda_{T,i}(m,d) + \lambda \mathcal{H}(u)\bigg)+ \bigg( \sum_{i=1}^Ku_{i} \tilde{R}^i_T(m,d)\bigg) \bigg\} +   \sum_{t=1}^{T-1} \beta^{t} \big(\smaxld \circ \alpha^\lambda_{T-1-t}\big)(m,d).
%\\
% &\leq \sup_{u \in \Delta^K}\bigg\{\sum_{i=1}^Ku_{i}\alpha^\lambda_{T,i}(m,d)+ \lambda \mathcal{H}(u)\bigg\} +   \sum_{t=1}^{T-1} \beta^{t} \big(\smaxld \circ \alpha^\lambda_{T-t}\big)(m,d) +  \max_{i\in [K]} \big| R^i_T(m,d) \big| \\
% & = \big( \smaxld \circ \alpha^\lambda_T \big) (m,d)  +   \sum_{t=1}^{T-1} \beta^{t} \big(\smaxld \circ \alpha^\lambda_{T-t}\big)(m,d) +  \max_{i\in [K]} \big| R^i_T(m,d) \big|
\end{align*}
Since $\sup_{u \in \Delta^K}\big| \sum_{i=1}^Ku_{i} \tilde{R}^i_T(m,d)\big| \leq Q_{\lambda,d}$, the inequality \eqref{eq: finite horizon value} for $T$ follows from \eqref{eq: smax}.  
\end{proof}

\subsection{Analysis of the regularised control problem over infinite horizon }
 We see that the error bound of our approximation in Theorem \ref{Thm: Expression of V_T} is uniform in $T$. We can now prove Theorem \ref{Thm: main bound for value function and approximation} by taking $T \to \infty$ in Theorem \ref{Thm: Expression of V_T}. 

\begin{proof}[Proof of Theorem \ref{Thm: main bound for value function and approximation}]
Fix  $(m,d) \in \Theta \times \cD$ and $\lambda >0$.

Since  $(\smaxld \circ \alpha^\lambda_t \big)(m,d) \to (\smaxld \circ \alpha^\lambda \big)(m,d)$ as $t \to \infty$, it follows from a Tauberian theorem (Theorem \ref{Thm: Tauberian}) that

\begin{equation}
\label{eq: smax convergence}
     \sum_{t=0}^{T-1} \beta^{t}\big(\smaxld \circ \alpha_{T-t}^\lambda\big)(m,d) \to (1-\beta)^{-1} \big(\smaxld \circ \alpha^\lambda \big)(m,d) \qquad \text{as } T \to \infty.
\end{equation}

Next, we will prove that $V^\lambda_T(m,d) \to   V^\lambda_\infty(m,d)	$ as $T \to \infty$. By (H.\ref{Assump: Dynamic Asumption}) together with the Cauchy--Schwarz inequality,  we can show that  for any $\varphi \in \Psi$,  $\sE_{m,d}|M^\varphi_{t+1} -M^\varphi_{t}| \leq C\|D^\varphi_t\| \leq C\|d\|$.   In particular,  $\sE_{m,d}|M^\varphi_{t}| \leq |m| + CT\|d\|$.   By (H.\ref{Assump: Instantaneous cost}), there exists a constant $C \geq 0$, such that
		\begin{equation}
		\left.
		    \label{eq: value converge horizon}
		\begin{aligned}
		&\sup_{\varphi \in \Psi}\sum_{t=T}^{\infty}\beta^{t} \sE_{m,d} \bigg[\bigg|\sum_{i=1}^K f_{i}\big(M^{\varphi}_t, D^{\varphi}_t \big) U^\varphi_{i,t+1}\bigg| \bigg] \leq C \sup_{\varphi \in \Psi}\sum_{t=T}^{\infty}\beta^{t} \sE_{m,d} \big[ |M^\varphi_t| + \|d\| + 1\big]
		\\
		& \hspace{5cm} \leq C \sum_{t=T}^\infty \beta^t \big( (t+1) \|d\| + |m| + 1\big) \to 0 \qq \text{as} \q T \to \infty. 
\end{aligned}	
\right\}
\end{equation}		
Moreover, since $\cH$ is a smooth entropy, $\cH$ is bounded and so  $\sup_{U \in \cU} \sum_{t=T}^\infty \beta^t \sE_{m,d} \big[ | \cH(U^\varphi_{t+1})| \big]  \to 0$ as $T \to 0$.   Combining this with \eqref{eq: value converge horizon}, we obtain
{\small
\begin{equation}
\label{eq: convergence of V_T}
    \Big|V_T^\lambda(m,d) - V^\lambda_\infty(m,d)\Big| \leq \sup_{\varphi \in \Psi}\bigg|\mathbb{E}_{m,d}\bigg[\sum_{t=T}^{\infty}\beta^{t} \Big( \sum_{i=1}^K f_{i}\big(M^{\varphi}_t, D^{\varphi}_t \big) U^\varphi_{i,t+1} + \lambda \mathcal{H}(U^\varphi_{t+1}) \Big)\bigg]\bigg| \to 0 \; \text{as} \; T \to\infty.
\end{equation}
}

Next, we will prove that $V^{\lambda, \varphi^{\lambda,T}}_{T}(m,d) \to V^{\lambda, \varphi^{\lambda}}(m,d)$ where $\varphi^{\lambda,  T}(t,m,d) := \del_a \smaxld \big( \alpha^\lambda(m,d) \big) $ and $\varphi^{\lambda,  T}(t,m,d) := \del_a \smaxld \big( \alpha^\lambda_{T-t+1}(m,d) \big) $ with $\alpha$ and $\alpha_t$ given \eqref{eq: alpha} and \eqref{eq: alpha_t}, respectively.

For any policy $\varphi \in \Psi$, let $g(t ; \varphi) := \mathbb{E}_{m,d}\big[ \big(\sum_{i=1}^K f_{i}\big(M^{\varphi}_t, D^{\varphi}_t \big) U^\varphi_{i,t+1}\big) + \lambda \mathcal{H}(U^\varphi_{t+1}) \big]$. By the similar argument as in \eqref{eq: value converge horizon}, we can find a constant $C \geq 0$ such that $\sup_{\varphi \in \Psi} g(t; \varphi) \leq C\big( (t+1) \|d\| + |m| + 1\big)$. Since $V^{\lambda, \varphi^{\lambda,T}}_{T}(m,d) = \sum_{t=0}^{T-1}\beta^{t} g(t ; \varphi^{\lambda,T})$ and $V^{\lambda, \varphi^{\lambda}}(m,d) = \sum_{t=0}^{\infty}\beta^{t} g(t ; \varphi^\lambda)$,  it suffices to show that for any fixed $t \in \sN$,  $g(t ; \varphi^{\lambda,T}) \to g(t ; \varphi^\lambda)$ as $T \to \infty$. Given this, the required convergence follows from a Tauberian theorem (Theorem \ref{Thm: Tauberian 2}).

Fix $t \in \sN$. By (H.\ref{Assump: Instantaneous cost}) and (H.\ref{Assump: Dynamic Asumption}), there exists a constant $C \geq 0$ such that $\big| L^\lambda(M^{\varphi^{\lambda },m,d}_s, D^{\varphi^{\lambda },m,d}_s)  \big|  \leq C$ for all $s \in \sN$. For any $s = 1,2, ...,t$, it follows from the mean value inequality that 
\begin{equation}
\label{eq: prob bound}
\left.
\begin{aligned}
&\Big| \nu^\lambda \big( \alpha^\lambda(M^{\varphi^{\lambda },m,d}_s, D^{\varphi^{\lambda },m,d}_s) \big) -\nu^\lambda \big( \alpha_{T-s+1}^\lambda(M^{\varphi^{\lambda },m,d}_s, D^{\varphi^{\lambda },m,d}_s) \big) \Big| \\ 
& \qquad \leq \|\partial_a \nu^\lambda\|_\infty \Big( \sum_{r=T-s + 1}^{\infty} \beta^{r}  \Big) \big| L^\lambda(M^{\varphi^{\lambda },m,d}_s, D^{\varphi^{\lambda },m,d}_s)  \big|  \\
&\qquad 
 \leq C \|\partial_a \nu^\lambda\|_\infty \beta^{T-s+1} (1-\beta)^{-1} \leq C \|\partial_a \nu^\lambda\|_\infty \beta^{T-t+1} (1-\beta)^{-1}. 
\end{aligned}
\right\}
\end{equation}
Fix $\epsilon \in (0,1)$ and choose $T_0$ such that, for all $T \geq T_0$,  we have $C \|\partial_a \nu^\lambda\|_\infty \beta^{T-t+1} (1-\beta)^{-1} < \epsilon/K$. 

As discussed in Section \ref{sec: setup}, we recall that the action corresponding to the policy $\varphi$ is given by $A^{\varphi, m,d}_s = \sup\big\{i : \sum_{k=1}^i U^{\varphi, m,d}_{s,k} \geq \zeta_s \big\}$.  By \eqref{eq: prob bound}, for $T \geq T_0$, the probability that the actions corresponding to $\varphi^\lambda$ and $\varphi^{\lambda,T}$ disagree at time $s$, provided they agree up to time $s-1$, satisfies
\begin{align*}
    &\sP(A^{\varphi^{\lambda},m,d}_s \neq A^{\varphi^{\lambda,T},m,d}_s |A^{\varphi^{\lambda},m,d}_r = A^{\varphi^{\lambda,T},m,d}_r \; \forall r = 1, 2,...,s-1 ) \\
    &\qquad \qquad \leq \sum_{i=1}^K \Big| \nu_i^\lambda \big( \alpha^\lambda(M^{\varphi^{\lambda },m,d}_s, D^{\varphi^{\lambda },m,d}_s) \big) -\nu_i^\lambda \big( \alpha_{T-s}^\lambda(M^{\varphi^{\lambda },m,d}_s, D^{\varphi^{\lambda },m,d}_s) \big) \Big| \leq \epsilon.
\end{align*}

Let $E_T$ be the the event that the actions corresponding to $\varphi^\lambda$ and $\varphi^{\lambda,T}$ agree up to time $t$, i.e.
$E_T := \{ A^{\varphi^{\lambda},m,d}_s = A^{\varphi^{\lambda,T},m,d}_s \; \forall s = 1,...,t \}$. This gives
$\sP(E_T) \geq (1-\epsilon)^t \geq 1 - t\epsilon$, i.e.  $\sP(E_T^c) \leq t \epsilon$.

By (H.\ref{Assump: Instantaneous cost}) and (H.\ref{Assump: Dynamic Asumption}), we can find a constant $C \geq 0$ such that for any $\varphi \in \Psi$, $$\sE_{m,d}|f\big(M^{\varphi}_t, D^{\varphi}_t \big)|^2 \leq Ct (|m|^2 + \|d\|^2 +1).$$ Moreover, as $\cH$ is bounded,  we can assume (wlog) that $\sup_{u\in \Delta^K}\cH(u) \leq C$. 

Since $(M^{\varphi^\lambda,m,d}_t, D^{\varphi^\lambda,m,d}_t) =  (M^{\varphi^{\lambda,T},m,d}_t, D^{\varphi^{\lambda,T},m,d}_t)$ and $U^{\varphi^\lambda, m,d}_t = U^{\varphi^{\lambda,T}, m,d}_t$ on $E_T$, we see that
$$|g(t;\varphi^{\lambda,T}) - g(t;\varphi^\lambda)| \leq 2\sup_{\varphi \in \Psi} \sE \Big[ \big(|f(M^\varphi_t, D^\varphi_t)|  + \lambda \cH(U^\Psi_{t+1})\big)\sI_{E_T^c} \Big] \leq 4Ct (|m|^2 + \|d\|^2 +2)(t\epsilon)$$
where the final inequality follows from the Cauchy--Schwarz inequality. 
As $\epsilon$ is arbitrary, we obtain the required result that for any fixed $t \in \sN$,  $g(t ; \varphi^{\lambda,T}) \to g(t ; \varphi^\lambda)$ as $T \to \infty$ which implies that
\begin{equation}
    \label{eq: value of policy convergence}
    V^{\lambda, \varphi^{\lambda,T}}_{T}(m,d) \to V^{\lambda, \varphi^{\lambda}}(m,d) \qquad \text{as } T \to \infty.
\end{equation}
Finally, combining \eqref{eq: smax convergence}, \eqref{eq: convergence of V_T}  and \eqref{eq: value of policy convergence} with Theorem \ref{Thm: Expression of V_T} gives the required error bound.
\end{proof}

\subsection{Analysis of the (unregularised) control problem}
This section is dedicated to prove Theorem \ref{Thm: error bound consistent}.

\begin{proof}[Proof of Theorem \ref{Thm: error bound consistent}]
Through the following argument, let $C \geq 0$ be a generic constant, depending only on $\beta$ and the bounds in (H.\ref{Assump: Instantaneous cost}) and (H.\ref{Assump: Dynamic Asumption}) which could be different between lines.  

 We will first show that there exists a constant $\tilde{C} \geq 0$ such that
 \begin{equation}
    \label{eq:diff Q value}
    V(m,d) - Q \big(\varphi^{\bs \lambda}(1,m,d), m, d \big) \leq \tilde{C}( \|d\|^{3-2\kappa} + \|d\|^{4-3\kappa} + \|d\|^{\kappa} ),
\end{equation}
where $Q(u,m,d) := \sum_{i=1}^K u_i \big(f_i(m,d) + \beta \sE\big[V\big(\Phi(m,d,i, \bs{\xi}_1) \big) \big] \big)$.

Observe that
{\footnotesize
$$V^{\bs{\lambda}(m,d), \varphi^{\bs{\lambda}(m,d)}}(m,d) = \sum_{i=1}^K  \varphi_i^{\bs{\lambda}}(1,m,d) \bigg( f_i(m,d)  + \beta  \sE \Big[V^{\bs{\lambda}(m,d), \varphi^{\bs{\lambda}(m,d)}}_\infty \big(\Phi(m,d,i, \bs{\xi}_1) \big)\Big] \bigg) + \bs{\lambda}(m,d) \cH\big(\varphi^{\bs{\lambda}}(1,m,d) \big)$$
}
and write
{\small
\begin{equation}
   \label{eq:decomposition q function} 
\left.
\begin{aligned}
    &Q \big(\varphi^{\bs \lambda}(1,m,d), m, d \big) - V(m,d)  \\
    &\quad = \Big( V^{\bs{\lambda}(m,d), \varphi^{\bs{\lambda}(m,d)}}(m,d) - V^{\bs{\lambda}(m,d)}(m,d) \Big) + \Big(V^{\bs{\lambda}(m,d)}(m,d) - V(m,d) \Big)   \\
    &\quad \quad + \beta \sum_{i=1}^K  \varphi_i^{\bs{\lambda}}(1,m,d) 
     \sE \Big[V^{\bs{\lambda}(m,d)}_\infty \big(\Phi(m,d,i, \bs{\xi}_1) \big) - V^{\bs{\lambda}(m,d), \varphi^{\bs{\lambda}(m,d)}}_\infty \big(\Phi(m,d,i, \bs{\xi}_1) \big)\Big]
     \\
    &\quad
    \quad + \beta \sum_{i=1}^K  \varphi_i^{\bs{\lambda}}(1,m,d) 
     \sE \Big[V \big(\Phi(m,d,i, \bs{\xi}_1) \big) - V^{\bs{\lambda}(m,d)}_\infty \big(\Phi(m,d,i, \bs{\xi}_1) \big)\Big] - \bs{\lambda}(m,d) \cH\big(\varphi^{\bs{\lambda}}(1,m,d) \big) 
\end{aligned}
\right\}
\end{equation}
}

Since $\cH$ is bounded, for any $\lambda > 0$,  $\lambda|\cH(u)| \leq C \lambda$ for all $u \in \Delta^K$ and $\big|V^\lambda (m,d) - V(m,d) \big| \leq C \lambda$ for all $(m,d) \in \Theta \times \cD$. Combining this observation with \eqref{eq:decomposition q function} and Theorem \ref{Thm: main bound for value function and approximation}, we see that
$$\Big|Q \big(\varphi^{\bs \lambda}(1,m,d), m, d \big) - V(m,d) \Big| \leq C \bs{\lambda}(m,d) + C \| d \|^3 \big(1 + \bs{\lambda}(m,d)^{-2} + \bs{\lambda}(m,d)^{-3} \|d\| \big).$$

Since $\cD$ is compact, $\|d\|^3 \leq C\|d\|^{3-2\kappa}$. Therefore,  \eqref{eq:diff Q value} follows from our assumption that $\underline{c} \|d\|^\kappa \leq \bs{\lambda}(m,d) \leq \bar{c} \|d\|^\kappa$ for some constant $\underline{c}, \bar{c} > 0$.

Let $\varphi^*:\sN \times \Theta \times \cD \to \Delta^K$ be the optimal (stationary) feedback policy of \eqref{Eq: no entropy reward}, i.e. $V^{\varphi^*} = V$. For each $N \in \sN$, let $$\varphi^{\bs{\lambda},N}(t,m,d) = \begin{cases}
\varphi^{\bs{\lambda}}(t,m,d) &; t = 1, 2, ..., N \\ \varphi^{*}(t,m,d) &; \text{otherwise}
\end{cases}. $$
We next show, by induction, that
$V(m,d) - V^{\varphi^{\bs{\lambda},N}}(m,d) \leq \eps(d)$ where $\eps(d) :=  (1-\beta)^{-1}\tilde{C}
( \|d\|^{3-2\kappa} + \|d\|^{4-3\kappa} + \|d\|^{\kappa} )$ with the constant $\tilde{C} \geq 0$ given in \eqref{eq:diff Q value}.

It is clear from the definition of $\varphi^{\bs{\lambda},N}$ that the required inequality holds for $N= 0$. Assume that the required inequality holds for $N-1$. Observe that
\begin{align*}
    V^{\varphi^{\bs{\lambda},N}}(m,d) &= \sum_{i=1}^K  \varphi_i^{\bs{\lambda}}(1,m,d) f_i(m,d) + \beta \sE_{m,d} \big[ V(M^{\varphi^{\bs \lambda}}_1 , D^{\varphi^{\bs \lambda}}_1 )\big] + \beta R_N(m,d) \\
    &= Q \big(\varphi^{\bs \lambda}(1,m,d), m, d \big) + \beta R_N(m,d),
\end{align*}
where $R_N(m,d) := \sE_{m,d}\big[ V^{\varphi_{N-1}^{\bs{\lambda}}}(M^{\varphi^{\bs \lambda}}_1 , D^{\varphi^{\bs \lambda}}_1 ) -  V(M^{\varphi^{\bs \lambda}}_1 , D^{\varphi^{\bs \lambda}}_1 )\big]$. By the induction hypothesis and (H.\ref{Assump: Dynamic Asumption}), $\|D^{\varphi^{\bs \lambda}}_1\| \leq \|d\|$, and thus
$|R_N(m,d)| \leq \sE[\eps(D^{\varphi^{\bs \lambda}}_1)] \leq \eps(d)$. Hence, the induction step follows from \eqref{eq:diff Q value}. 

The main inequality of this theorem follows from the fact that $ V^{\varphi^{\bs{\lambda},N}}(m,d) \to  V^{\varphi^{\bs{\lambda}}}(m,d)$ as $N \to \infty$ which can be proved in the similar manner as the proof that $V^\lambda_T(m,d) \to   V^\lambda_\infty(m,d)$ as $T \to \infty$ in Theorem \ref{Thm: main bound for value function and approximation}.
\end{proof}

\subsection{Complete Learning}
We have discussed in Section \eqref{sec: limit} that some bandit algorithms e.g. greedy or KG may suffer from incomplete learning. 
Theorem \ref{Thm: Complete Learning for bounded}
show that the ARC algorithm overcomes this limitation, in the sense that $\|D^{\varphi^{\bs\lambda}}_t \| \to 0$ as $t \to \infty$.  This section is dedicated to prove this result. 

\begin{proof}[Proof of Theorem \ref{Thm: Complete Learning for bounded}]
By (H.\ref{Assump: Instantaneous cost}), (H.\ref{Assump: Dynamic Asumption}) and the fact that $\bs{\lambda}(m,d) \geq \underline{c} \|d\|^\kappa$, there exists a constant $C \geq 0$ such that $|L^{\bs{\lambda}(m,d)}(m,d)| \leq C\big(1+\bs{\lambda}(m,d)^{-1} \big)\|d\|^2 \leq C( \|d\|^2 + \|d\|^{2-\kappa}\underline{c}^{-1})$. Since $\kappa \leq 2$ and $\cD$ is compact,  $L^{\bs{\lambda}(m,d)}(m,d)$ is uniformly bound on $\Theta \times \cD$. Combining this with the boundedness of $f$ yields $\sup_{(m,d) \in \Theta \times \cD}|\alpha^{\bs{\lambda}(m,d)}(m,d)| \leq C$ for some constant $C \geq 0$. 

Fix  $(m,d) \in \Theta \times \cD$, $i \in [K]$ and $\epsilon > 0$ and consider the events $E_\epsilon := \{ \|D^{\varphi^{\bs{\lambda}},m,d}_t  \| > \epsilon \q  \forall t \in \sN \}$, $F_i := \{ A^{\varphi^{\bs{\lambda}},m,d}_s = i \; \text{for finitely many } t \in \sN \}$ and $G_{i,t} := \{A^{\varphi^{\bs{\lambda}},m,d}_t \neq i \}$. 

For $\|d\| > \epsilon$, $\alpha^{\bs{\lambda}(m,d)}(m,d)/\bs{\lambda}(m,d)$ takes values in a compact set.  Hence, there exists $r \in (0,1)$ such that
$\nu^{\bs{\lambda}(m,d)}_i\big(\alpha^{\bs{\lambda}(m,d)}(m,d)\big) = (\del_a S)_i \big(\alpha^{\bs{\lambda}(m,d)}(m,d)/\bs{\lambda}(m,d)\big) > r$ for $\|d\| > \epsilon$. In particular, for any $n,N \in \sN$,
$
    \sP\big( \bigcap_{t=n}^N G_{i,t} \big| E_\epsilon \big) =  \prod_{t=n}^N \sP\big(  G_{i,t} \big| E_\epsilon  \cap \bigcap_{s=n}^{t-1} G_{i,s} \big) \leq  (1-r)^{N-n+1} \to 0
$ as $N \to \infty$.
Therefore, $$\sP(F_i \cap E_\epsilon) = \sP(E_\epsilon) \sP(F_i|E_\epsilon) =\sP(E_\epsilon) \sP\bigg( \bigcup_{n=1}^\infty \bigcap_{t=n}^\infty G_{i,t} \bigg| E_\epsilon \bigg) = \sP(E_\epsilon)\lim_{n \to \infty}  \sP\bigg( \bigcap_{t=n}^\infty G_{i,t} \bigg| E_\epsilon \bigg) \leq 0.$$
 We can deduce from assumption $(i)$ that
 $E_\epsilon \subseteq \cup_{i \in [K]} F_i$. Therefore, $\sP(E_\epsilon) = \sP(\bigcup_{i \in [K]} F_i \cap E_\epsilon) = 0$. The required result follows by considering the event $\bigcup_{n=1}^\infty E_{1/n}$.
\end{proof}

\appendix

\section{Proofs of relevant results}
\label{sec:Proofs of relevant results}

\begin{Theorem}[Robust Representation]
	\label{thm: robust rep}
	A convex  nonlinear expectation $S$ admits a representation of the form
	$S(a) = \sup_{u \in \Delta^K} \Big(\sum_{i=1}^K u_i a_i + \mathcal{H}_{\max}(u)\Big),$
	where
	$\mathcal{H}_{\max}(u) := - \sup_{a \in \mathcal{A}_S} \Big(\sum_{i=1}^Ku_i a_i\Big)$ and $\mathcal{A}_S:= \{a \in \mathbb{R}^K : S(a) \leq 0\}.$
	
	Furthermore, $\mathcal{H}_{\max}$ is the maximal function which represents $S$, i.e. if there exists $\mathcal{H}$ such that \eqref{eq: robust rep} holds with $\mathcal{H}$, then $\mathcal{H}(u) \leq \mathcal{H}_{\max}(u)$ for all $u \in \Delta^K$.
	\begin{proof}
	See F\"ollmer and Schied \cite[Theorem 4.16]{stoc_fin} or Frittelli and Rosazza Gianin \cite{Frittelli_and_Rosazza}.
	\end{proof}
\end{Theorem}

\begin{proof}[Proof of Theorem \ref{thm: equivalence to boundedness}] {$(i) \Rightarrow (ii)$ :} Fix $i \in [K]$. Consider $a = (N+\epsilon)e_i+ \sum_{j \neq i} r_j e_j$ where $r_j \in \mathbb{R}$ for all $j \neq i$ and $e_i$ is the $i$-th basis vector in $\sR^K$. By $(i)$, $S(a) + N \geq \max(N+\epsilon, r_j) \geq N+\epsilon$. Hence, $S(a) \geq \epsilon > 0$.  
	
	As $\epsilon$ is arbitrary, it follows that $\mathbb{R}^{i-1} \times (N, \infty) \times \mathbb{R}^{K-i} \subseteq \mathcal{A}^c_S$. The result then follows by considering intersection over all $i$.
	
	\paragraph{$(ii) \Rightarrow (iii)$:} By Theorem \ref{thm: robust rep}, we can write 
	$S(a) = \sup_{u \in \Delta^K} \Big(\sum_{i=1}^K u_i a_i + \mathcal{H}_{max}(u)\Big),$
	where
	$\mathcal{H}_{max}(u) := - \sup_{a \in \mathcal{A}_S} \Big(\sum_{i=1}^Ku_i a_i\Big) $ with $\mathcal{A}_S:= \{a \in \mathbb{R}^K : S(a) \leq 0\}.$
	
	As $\mathcal{A}_S \subseteq (-\infty, N]^K$ and $u \in \Delta^K$,
	$\mathcal{H}_{max}(u) \geq - \sup_{a \in (-\infty, N]^K} \Big(\sum_{i=1}^Ku_i a_i\Big) \geq -N.$
	
	Moreover, by \eqref{eq: robust rep}, we have $\sup_{u \in \Delta^K}\mathcal{H}_{max}(u) \leq S(0)$. Therefore,   $\mathcal{H}_{max}$ is bounded.
	
	\paragraph{$(iii) \Rightarrow (iv)$ :} Fix $a \in \mathbb{R}^K$ and define $i^* \in \argmax_i a_i$. Then
	\begin{align*}
	&- \lambda\sup_{u \in \Delta^K}|\mathcal{H}(u)| \leq \lambda \mathcal{H}(e^{(i^*)}) = \sum_{i=1}^K(e^{(i^*)})_ia_i + \lambda \mathcal{H}(e^{(i^*)}) - \max_i a_i  \\
	&  \qquad \leq \smaxld(a) - \max_i a_i \leq \sup_{u \in \Delta^K} \Big(\sum_{i=1}^K u_i a_i\Big) + \lambda\sup_{u \in \Delta^K}|\mathcal{H}(u)| - \max_i a_i = \lambda \sup_{u \in \Delta^K}|\mathcal{H}(u)|.
	\end{align*}
	Hence, $\sup_{a \in \mathbb{R}} |\smaxld(a) - \max_i a_i| \leq \lambda \sup_{u \in \Delta^K}|\mathcal{H}(u)| \to 0$ as $\lambda \downarrow 0$.
	
	\paragraph{$(iv) \Rightarrow (i)$} Find $N > 0$ such that
	$$1 \geq \sup_{a \in \mathbb{R}} |S^{1/N}_{\text{max}}(a) - \max_i a_i| = \frac{1}{N}\sup_{a \in \mathbb{R}} |S(Na) - \max_i Na_i| = \frac{1}{N}\sup_{a \in \mathbb{R}} |S(a) - \max_i a_i|.$$
	By rearranging the inequality above, the result follows.
\end{proof}

\begin{proof}[Proof of Lemma \ref{lem: error of L}]  Through the following argument, let $C \geq 0$ be a generic constant, depending only on $\beta$ and the bounds in (H.\ref{Assump: Instantaneous cost}) and (H.\ref{Assump: Dynamic Asumption}) which could be different between lines.  

Recall expressions for $\cB^\lambda$, $\cM^\lambda$, $\Sigma^\lambda$ from \eqref{eq: B, M, Sigma Expression}. Since $f$ is $3$-times differentiable with bounded derivatives (H.\ref{Assump: Instantaneous cost}),  the terms $f$, $\del_m f$, $\del^2_m f$ and $\del_d f$ are differentiable with bounded derivative.  Moreover, since $\cH$ is a smooth entropy,  the corresponding smooth max approximator $S$ has a bounded derivative.  In particular, $\big| \del_a \nu^\lambda(a) \big| \leq C/\lambda$ and $ \big| \del_a \eta^\lambda(a)\big| \leq C/\lambda$.  For $(\Delta m_i , \Delta d_i) :=  \Phi(m,d, i, \bs{\xi}_t) - (m,d)$, it follows from the mean value inequality that 
	\begin{align*}
	\big| \mathcal{B}^\lambda(m + \Delta m_i ,d + \Delta d_i) -  \mathcal{B}^\lambda(m,d) \big| &\leq C(1+\lambda^{-1}) \big(|\Delta m_i | + \|\Delta d_i \| \big),  &&| \mathcal{B}^\lambda| \leq C, \\
\big|	\mathcal{M}^\lambda(m + \Delta m_i ,d + \Delta d_i) - \mathcal{M}^\lambda(m,d) \big| &\leq C(1+\lambda^{-1}) \big(|\Delta m_i | + \|\Delta d_i \| \big),  &&|\mathcal{M}^\lambda| \leq C, \\
	\big| \Sigma^\lambda(m + \Delta m_i ,d + \Delta d_i) - \Sigma^\lambda(m,d) \big| &\leq C(1+\lambda^{-2}) \big(|\Delta m_i | + \|\Delta d_i \| \big),  \; \text{and} &&|\Sigma^\lambda| \leq C(1+\lambda^{-1}).
	\end{align*}
	By similar arguments as above applying to (H.\ref{Assump: Dynamic Asumption}), we show that for any  $\psi \in \big\{b_i, \mu_i, \big(\sigma_i\sigma_i^\top\big)  :i \in \mathcal{A}\big\}$, and $(m,d) \in  \Theta \times \cD$, $|\psi(m,d)| \leq C \|d\|^2$.
	\begin{align*}
	&\big| \psi(m+\Delta m_i, d + \Delta d_i) - \psi(m,d) \big| \leq   \sup_{\tilde{d} \in [d, d+ \Delta d_i], \tilde{m} \in \Theta}
	\Big( |\del_d \psi(m,\tilde{d})| \cdot \| \Delta d_i \|  +  |\del_m \psi(\tilde{m},d)| \cdot  |\Delta m_i| \Big)\\
	&\qquad \leq C \sup_{\tilde{d} \in [d, d+ \Delta d_i], \tilde{m} \in \Theta}
	\Big( \|\tilde{d} \| \cdot \| \Delta d_i \|  +  \|d\|^2 \cdot  |\Delta m_i| \Big) \leq C \big( \| d\| \cdot \| \Delta d_i \|  +  \|d\|^2 \cdot  |\Delta m_i| \big)
	\end{align*}
	where $[d, d+ \Delta d_i]$ is defined to be a rectangle in $\sR^q$ and the final inequality follows from the convexity of the norms and (H.\ref{Assump: Dynamic Asumption})$(i)$.
	
Substituting the above inequalities into \eqref{eq: L expression},  we obtain
	\begin{align*}
	&\big| L_i(m + \Delta m_i ,d + \Delta d_i) - L_i(m,d) \big| \leq C(1+\lambda^{-2}) \Big( \big(|\Delta m_i |\cdot  \|d\|^2 + \|\Delta d_i \| \cdot  \|d\|^2 + \| d\|\cdot \| \Delta d_i \|   \Big).
	\end{align*}
Finally,  by  (H.\ref{Assump: Dynamic Asumption})$(iii)-(iv)$ and Cauchy--Schwartz  inequality, $\sE |\Delta m_i | \leq C \|d\|$ and $ |\Delta d_i | \leq C \|d\|^2$. Substituting these bounds into the above inequality, the general result follows.
\end{proof}

\begin{proof}[Proof of Lemma \ref{lem: one step return error}] Let  $g(m,d) = f(m,d) + c$ where $c \in \sR$ is a given constant.  By (H.\ref{Assump: Instantaneous cost}) and Definition \ref{def: smooth entropy},  $\smaxld \circ g$ is 3-times differentiable. Consider the Taylor's approximation 
\begin{equation}
\label{eq: first estimate}
\left.
\begin{aligned}
& \big(\smaxld\circ g \big) \big( \Phi(\cdot,\cdot, i, \bs{\xi}_t) \big) - \big(\smaxld\circ g\big)\\
&\quad =  \big\langle \del_d \big(\smaxld \circ g \big) ;\Delta d_i  \big\rangle  +  \big\langle \del_m \big(\smaxld \circ g \big) ; \Delta m_i \big\rangle + \frac{1}{2} \big\langle \del^2_{m} \big(\smaxld \circ g \big) ; \Delta m_i  \Delta m_i ^\top \big\rangle + \Delta S^1_T
\end{aligned}
\right\}
\end{equation}
where all derivatives are evaluated at $(m,d)$ and $\Delta S^1_T$ denotes the remaining terms.  

By (H.\ref{Assump: Instantaneous cost}) and Definition \ref{def: smooth entropy},  the second and third derivatives of $\big(\smaxld \circ g\big)(m,d) $ are $ \cO(1 + \lambda^{-2})$.  Moreover, by (H.\ref{Assump: Dynamic Asumption}), $\sE | \Delta m_i | \leq C \|d\|$ and $ \|\Delta d_i  \|  \leq C \|d\|^2$.  Applying these bounds to the third order terms and the remaining second order term, we obtain $\sE | \Delta S^1_T | \leq C(1+\lambda^{-2}) \|d\|^3 $.  Here, the bounded constant $C \geq 0$ is uniform over $c \in \sR$.

Taking the expectation of \eqref{eq: first estimate},  we see that 
\begin{equation}
\label{eq: first estimate expectation}
\left.
\begin{aligned}
&\sE \big[ \big(\smaxld\circ g \big) \big( \Phi(\cdot,  \cdot, i, \bs{\xi}_t) \big) \big] - \big(\smaxld\circ g\big)\\
&\qquad =  \big\langle \del_d \big(\smaxld \circ g \big) ; b_i  \big\rangle  + \big\langle \del_m \big(\smaxld \circ g \big) ; \mu_i \big\rangle + \frac{1}{2} \big\langle \del^2_{m} \big(\smaxld \circ g \big) ; \sigma_i\sigma_i^\top \big\rangle +  R^1_T
\end{aligned}
\right\}
\end{equation}
where $|R^1_T| \leq C(1+\lambda^{-2})\|d\|^3$.  Now write
 \begin{equation}
 \label{eq: derivative S a}
 \left.
 \begin{aligned}
 \del_d \big(\smaxld \circ g\big)&= \sum_{i=1}^K \big(\del_a \smaxld \circ g\big)_i \big(\del_d f_i \big), \qq  \del_m \big(\smaxld \circ g\big):= \sum_{i=1}^K \big(\del_a \smaxld \circ g\big)_i \big(\del_m f_i \big)
\\	
\del^2_{m} \big(\smaxld \circ g\big) &= \sum_{i=1}^K\big(\del_a \smaxld \circ g\big)_i \big(\del^2_m f_i \big) + \sum_{i,j=1}^K\big(\del^2_a \smaxld \circ g\big)_{ij} \big(\del_m f_i \big) \big(\del_m f_i \big)^\top.
 \end{aligned}
 \right\}
 \end{equation}
From Definition \ref{def: smooth entropy},  $\del^k_a \smaxld(a) = \cO(\lambda^{1-k})$ for $k = 1,2,3$.  By (H.\ref{Assump: Instantaneous cost}),  any terms involving derivatives of $f$ in \eqref{eq: derivative S a} are uniformly bounded.  By the mean value theorem and (H.\ref{Assump: Dynamic Asumption}), \eqref{eq: first estimate expectation} yields
\begin{equation}
\label{eq: first estimate expectation final}
\left.
\begin{aligned}
&\sE \big[ \big(\smaxld\circ g \big) \big( \Phi(\cdot, \cdot, i, \bs{\xi}_t) \big) \big] - \big(\smaxld\circ g\big) \\
&\quad =  \big\langle \del_d \big(\smaxld \circ f \big) ; b_i  \big\rangle  + \big\langle \del_m \big(\smaxld \circ f \big) ; \mu_i \big\rangle + \frac{1}{2} \big\langle \del^2_{m} \big(\smaxld \circ f \big) ; \sigma_i\sigma_i^\top \big\rangle + R^2_T = L_i^\lambda +R^2_T
\end{aligned}
\right\}
\end{equation}
where $|R^2_T| \leq C(1+\lambda^{-2})\|d\|^3 + |c|(1+\lambda^{-2})\|d\|^2$.

Denote $a(m,d, i, \bs{\xi}_t) := f \big( \Phi(m,d, i, \bs{\xi}_t) \big) +  L^\lambda(m,d) \big( \sum_{s=1}^{T-1}\beta^s \big)$. Consider $c = L^\lambda(m,d) \big( \sum_{s=1}^{T-1}\beta^s \big)$. We see that $|c| \leq C (1+\lambda^{-1})\|d\|^2$. Therefore, \eqref{eq: first estimate expectation final} yields
\begin{equation}
\label{eq: first estimate use}
 \Big|\sE \big[ \smaxld \big(a(\cdot, \cdot, i, \bs{\xi}_t) \big) \big]- \big(\smaxld\circ \alpha^\lambda_T\big) - L_i^\lambda \Big| \leq C \|d\|^3 \Big(1+ \lambda^{-2} + \lambda^{-3} \|d\| \Big).
\end{equation}

 Since the first derivative of $\smaxld $ is bounded and does not depend on $\lambda$, it follows from the mean value theorem that
 {\small
 \begin{equation}
     \label{eq: second estimate use}
\left.
\begin{aligned}
& \sE \bigg|  \smaxld  \Big( \alpha^\lambda_T \big( \Phi(m,d, i, \bs{\xi}_t) \big) \Big) - \smaxld  \Big(a(m,d, i, \bs{\xi}_t) \Big)  \bigg| \leq C \sE \big| \alpha^\lambda_T \big( \Phi(m,d, i, \bs{\xi}_t)\big)  - a(m,d, i, \bs{\xi}_t) \big|  \\
&\qquad = C  \sE \bigg|\Big( \sum_{s=1}^{T-1}\beta^s \Big)L^\lambda \big( \Phi(m,d, i, \bs{\xi}_t) \big) -\Big( \sum_{s=1}^{T-1}\beta^s \Big) L^\lambda(m,d) \bigg| \leq C(1-\beta)^{-1}(1+\lambda^{-2}) \|d\|^3 
\end{aligned}
\right\}
 \end{equation}
 }
where the final inequality follows from Lemma \ref{lem: error of L}.  Combining \eqref{eq: first estimate use} and \eqref{eq: second estimate use},  the result follows.

%If in addition (H.\ref{Assump: Linear}) holds, then the upper-bound for \eqref{eq: second estimate use} is replaced by $C(1-\beta)^{-1}P_{\lambda,d}(1,3)$ which modify our upper-bound.
\end{proof}

\begin{Theorem}[Tauberian theorem 1]
	\label{Thm: Tauberian}
	Let $(a_t)$ be a real-value sequence converging to $a$ and $\beta \in (0,1)$.  Then $\sum_{t=0}^{T-1} \beta^t a_{T-t} \to (1-\beta)^{-1} a$ as $T \to \infty$. 
	\begin{proof}
		Fix $\epsilon > 0$ and find $s > 0$ such that for all $t \geq s$, $|a_t-a| \leq \epsilon$. 
		Since $(a_t)$ is also bounded, we can find $T_0 > s$ such that for any $t > T_0-s$,  $\beta^{t/2} |a_u- a| \leq \epsilon$ for all $u \in \sN$.  Hence, for any $T \geq T_0$,
		\begin{align*}
		& \Big| \sum_{t=0}^{T-1} \beta^t a_{T-t} - (1-\beta)^{-1}a\Big| = \Big|\sum_{t=0}^{T-1} \beta^{t}(a_{T-t} -a) + \beta^T (1-\beta)^{-1}a\Big| \\
		& \qq \leq \sum_{t=T-s}^{T-1 }\beta^{t} |a_{T-t}-a| + \sum_{t=s+1}^{T} \beta^{T-t} |a_t-a|  + \beta^T (1-\beta)^{-1}|a| \\
		&\qq \leq  \sum_{t=T-s}^{T-1 }\beta^{t/2} \epsilon + \sum_{t=s+1}^{T} \beta^{T-t}\epsilon  + \beta^T (1-\beta)^{-1}|a| \leq (1-\sqrt{\beta})^{-1}\epsilon + (1-\beta)\epsilon +  \beta^T (1-\beta)^{-1}|a|.
		\end{align*}
Hence, $\limsup_{T \to \infty}\big| \sum_{t=0}^{T-1} \beta^t a_{T-t} - (1-\beta)^{-1}a\big| \leq (1-\sqrt{\beta})^{-1}\epsilon + (1-\beta)\epsilon $.  Since $\epsilon$ is arbitrary, we obtain  the required result.
	\end{proof}
\end{Theorem}
  \begin{Theorem}[Tauberian theorem 2]
\label{Thm: Tauberian 2}
	Let $g: \sN_0 \times \sN_0 \to \sR$ and $g^*: \sN_0 \to \sR$ be functions with a constant $C \geq 0$ such that $|g(t,T)| \leq C t $ and $g(t,T) \to g^*(t)$ as $T \to \infty$ for all $t \in \sN$.  Then for any $\beta \in (0,1)$,  $\sum_{t=0}^{T-1} \beta^t g(t,T) \to \sum_{t=0}^\infty \beta^t g^*(t)$ as $T \to \infty$. 
	\begin{proof}
		Fix $\epsilon > 0$ and find $N > 0$ such that for all $t \geq N$ and $T \in \sN$,  $\beta^{t/2} |g(t,T) - g^*(t)| \leq \epsilon$  where such $N$ exists due to the linear growth condition. 
		
	Since $g(t,T) \to g^*(t)$ as $T \to \infty$ , we can find $T_0 \geq N$ such that for any $T \geq T_0$,  $|g(t,T) - g^*(t)| \leq \epsilon$ for all $t = 0, 1,...,N-1$.   For any $T \geq T_0$, 
	
		\begin{align*}
		& \Big|  \sum_{t=0}^{T-1} \beta^t g(t,T)-\sum_{t=0}^\infty \beta^t g^*(t)\Big|  \leq \sum_{t=0}^{N-1}\beta^{t} |g(t,T) - g^*(t)| + \sum_{t=N}^{T-1} \beta^{t}  |g(t,T) - g^*(t)|  + \sum_{t=T}^{\infty} \beta^t |g^*(t)|\\
		& \qq \leq \sum_{t=0}^{N-1}\beta^{t} \epsilon + \sum_{t=N}^{T-1} \beta^{t/2} \epsilon  + C \sum_{t=T}^{\infty} t \beta^t \leq (1-\beta)^{-1} \epsilon  + (1-\sqrt{\beta})^{-1} \epsilon +T \beta^{T-1}(1-\beta)^{-2}.
	\end{align*}
Hence, $\limsup_{T \to \infty}\big| \sum_{t=0}^{T-1} \beta^t a_{T-t} - (1-\beta)^{-1}a\big| \leq (1-\sqrt{\beta})^{-1}\epsilon + (1-\beta)\epsilon $.  Since $\epsilon$ is arbitrary, we obtain  the required result.
	\end{proof}
\end{Theorem}
%\begin{Lemma} For $m \in \sR^p$ and $d \in \sR^{q \times q}$. Let $f(m,d) = \int_{\sR} g(b^\top m + (b^\top d b+ \sigma^2)^{1/2} z)\varphi(z) \d z$ where $\varphi(z) = (2\pi)^{-1/2} \exp(-z^2/2)$.  Then
%$$\del^2_m f(m,d) = \q \text{and} \q \del_d f(m,d) =$$
%\begin{proof} It is easy to see that $\del^2_m f(m,d) = b b^\top \int_\sR  g''(b^\top m + (b^\top d b+ \sigma^2)^{1/2} z)\varphi(z) \d z$ and $\del^2_m f(m,d) = b b^\top \int_\sR  g'(b^\top m + (b^\top d b+ \sigma^2)^{1/2} \varphi(z) \d z$

%\end{proof}
%\end{Lemma}
\begin{Lemma}
\label{lem: general derivative}
Consider the case when $m \in \sR^p$ and $d \in \sS^p_+$.  Suppose that  $\sigma_i \sigma_i^\top(m,d) = - b_i (m,d)$, $\mu_i(m,d) = 0$ and $f_i(m,d) = \int_{\sR} r_i(c_i^\top m + (c_i^\top d c_i+ \rho_i^2)^{1/2} z)\varphi(z) \d z$ where $r_i : \sR \to \sR$ is sufficiently smooth and $\varphi(z) = (2\pi)^{-1/2} \exp(-z^2/2)$. Then for every $i \in [K]$,
$$L_i^\lambda(m,d) =\frac{1}{2 \lambda} \Big\langle \sum_{j,k =1}^K\Big(\eta^{\lambda}_{jk}\big((f(m,d) \big) \big(\del_m f_j(m,d) \big) \big(\del_m f_k(m,d) \big)^\top\Big) ; \sigma_i \sigma_i^\top (m,d) \Big\rangle.$$
\end{Lemma}
\begin{proof}
Since $r_i$ is sufficiently smooth,   $\del^2_m f(m,d) = c_i c_i^\top \int_\sR  r_i''(c_i^\top m + (c_i^\top d c_i+ \rho_i^2)^{1/2} z)\varphi(z) \d z$ and $\del_d f(m,d) = \frac{1}{2}(c_i^\top d c_i+ \rho_i^2)^{-1/2}c_i c_i^\top \int_\sR  r_i'(c_i^\top m + (c_i^\top d c_i+ \rho_i^2)^{1/2} z) \varphi(z) z\d z$. By integration by parts, $\del^2_m f_i(m,d) = \frac{1}{2}\del^2_d f_i(m,d)$.  Substituting this into the expression for $L^\lambda$, the result follows.
\end{proof}
\begin{Lemma}
\label{lem: add info derive} 
For the bandit with additional information described in Section \ref{sec: address limit}, we have 
$$L_i^\lambda(m,d) =\frac{1}{2 \lambda} \sum_{j = 1}^K \eta^\lambda_{jj} \big(f(m,d)\big) d_{jj}^2\Big(\frac{s_{ij}}{1+ d_{jj}s_{ij}}\Big)  g_j^2(m,d)$$
where $g_i(m,d) =  \int_{\sR} r'_i(e_i^\top m + (e_i^\top d e_i+ s_{ii}^{-1})^{1/2} z)\varphi(z) \d z$.
\begin{proof}We note that $f_i(m,d) = \int_{\sR} r_i(e_i^\top m + (e_i^\top d e_i+ s_{ii}^{-1})^{1/2} z)\varphi(z) \d z$ and
 $$\sigma_i \sigma_i^\top(m,d) = - b_i(m,d) = \text{Diag}\bigg(d_{11}^2 \Big(\frac{s_{i1}}{1+ d_{11}s_{i1}}\Big), ..., d_{KK}^2 \Big(\frac{s_{iK}}{1+ d_{KK}s_{iK}}\Big) \bigg).$$
Substituting these into Lemma \ref{lem: general derivative}, the result follows.
\end{proof}
\end{Lemma}

\begin{Lemma}
\label{lem: linear derive} 
For the one-dimensional linear bandit with $Y^{(i)} \sim N(c_i^\top\bs{\theta}, P^{-1}_i)$ and prior $\bs{\theta} \sim N(m,d)$ with Shannon entropy (Remark \ref{Rem: Shannon Entropy}), we have
{\small
$$L_i^\lambda(m,d) = \frac{( c_i^\top d c_i + P_i^{-1})^{-1}}{2 \lambda}  \bigg( \sum_{j=1}^K \nu_j^\lambda \big((f(m,d) \big) g_j(m,d)^2 \big(c_i^\top d c_j \big)^2 - \bigg(\sum_{j=1}^K \nu_j^\lambda \big((f(m,d) \big) g_j(m,d) \big(c_i^\top d c_j \big) \bigg)^2 \bigg)$$}
where $g_i(m,d) =  \int_{\sR} r'_i(c_i^\top m + (c_i^\top d c_i+ P_i^{-1})^{1/2} z)\varphi(z) \d z$.
\end{Lemma}
\begin{proof}We note that $f_i(m,d) = \int_{\sR} r_i(c_i^\top m + (c_i^\top d c_i+ P_i^{-1})^{1/2} z)\varphi(z) \d z$ and we have
\begin{align*}
\sigma_i(m,d) &=  dc_i\Big(1- \frac{c_i^\top dc_i}{P_i^{-1} + c_i^\top d c_i} \Big) P_i \big( c_i^\top d c_i + P_i^{-1}\big)^{1/2} = dc_i\big( c_i^\top d c_i + P_i^{-1}\big)^{-1/2}.
\end{align*}
In particular,
 $\sigma_i\sigma_i(m,d)^\top = - b_i(m,d) = dc_i c_i^\top d ( c_i^\top d c_i + P_i^{-1})^{-1}.$ Substituting these into Lemma \ref{lem: general derivative},  we can write $L_i^\lambda(m,d) =\frac{1}{2 \lambda} \big\langle \sum_{j,k =1}^K\big(\eta^{\lambda}_{jk}\big((f(m,d) \big) g_i(m,d) g_j(m,d) c_i c_j^\top\big) ; \sigma_i \sigma_i^\top (m,d) \big\rangle$ and obtain the required result.
\end{proof}

\bibliographystyle{siam}

\bibliography{bibliography.bib}

\begin{thebibliography}{10}

\bibitem{Original_UCB}
{\sc R.~Agrawal}, {\em {Sample mean based index policies by $O(\log N)$ regret
  for the multi-armed bandit problem}}, Advances in Applied Probability,
  (1995), pp.~1054--1078.

\bibitem{UCB_tuned}
{\sc P.~Auer, N.~Cesa-Bianchi, and P.~Fischer}, {\em Finite-time analysis of
  the multiarmed bandit problem}, Machine Learning,  (2002), p.~235–256.

\bibitem{Brezzi_and_Lai_approx}
{\sc M.~Brezzi and T.~Lai}, {\em Optimal learning and experimentation in bandit
  problems}, Journal of Economic Dynamics and Control,  (2002), pp.~87--108.

\bibitem{adaptive_learning_survey}
{\sc G.~Burtini, J.~Loeppky, and R.~Lawrence}, {\em {A Survey of Online
  Experiment Design with the Stochastic Multi-Armed Bandit}},
  arXiv:1510.00757v4,  (2015).

\bibitem{BGE_Alg}
{\sc N.~Cesa-Bianchi, C.~Gentile, G.~Lugosi, and G.~Neu}, {\em {Boltzmann
  Exploration Done Right}}, in NIPS'17: Proceedings of the 31st International
  Conference on Neural Information Processing Systems, 2017.

\bibitem{Gaussian_maximal}
{\sc V.~Chernozhukov, D.~Chetverikov, and K.~Kato}, {\em {Comparison and
  anti-concentration bounds for maxima of Gaussian random vectors}},
  Probability Theory and Related Fields,  (2013).

\bibitem{DR_original}
{\sc S.~N. Cohen}, {\em Data-driven nonlinear expectations for statistical
  uncertainty in decisions}, Electronic Journal of Statistics,  (2016),
  pp.~1858--1889.

\bibitem{Robust_Gittins_Our_paper}
{\sc S.~N. Cohen and T.~Treetanthiploet}, {\em {Gittins' theorem under
  uncertainty}}, Electronic Journal of Probability,  (2022), pp.~1--48.

\bibitem{represent_nonlinear}
{\sc F.~Coquet, Y.~Hu, J.~M\'emin, and S.~Peng}, {\em {Filtration consistent
  nonlinear expectations and related $g$-expectations}}, Probability Theory and
  Related Fields,  (2002), pp.~1--27.

\bibitem{Greedy_regret1}
{\sc E.~Even-Dar, S.~Mannor, and Y.~Mansour}, {\em Action elimination and
  stopping conditions for the multi-armed bandit and reinforcement learning
  problems}, Journal of Machine Learning Research,  (2006).

\bibitem{GLM_bandit}
{\sc S.~Filippi, O.~Capp\'e, A.~Garivier, and C.~Szepesv\'ari}, {\em
  {Parametric bandits: the Generalized Linear case}}, in {NIPS'10: Proceedings
  of the 23rd International Conference on Neural Information Processing
  Systems}, 2010.

\bibitem{stoc_fin}
{\sc H.~F\"ollmer and A.~Schied}, {\em {Stochastic Finance: an introduction in
  discrete time}}, De Gruyler, 2016.

\bibitem{Frittelli_and_Rosazza}
{\sc M.~Frittelli and E.~R. Gianin}, {\em Putting order in risk measures},
  Journal of Banking \& Finance,  (2002), pp.~1473--1486.

\bibitem{Gittins_book}
{\sc J.~Gittins}, {\em Multi-armed Bandit Allocation Indices}, John Wiley and
  Sons, 1989.

\bibitem{Gittin_origin}
{\sc J.~C. Gittins and D.~M. Jones}, {\em A dynamic allocation index for the
  sequential design of experiments}, in Progress in Statistics, J.~Gani, ed.,
  Amsterdam: North Holland, 1974, pp.~241--266.

\bibitem{Bayes_UCB}
{\sc E.~Kaufmann, O.~Capp\'e, and A.~Garivier}, {\em {On Bayesian Upper
  Confidence Bounds for Bandit problems}}, in Artificial intelligence and
  statistics, 2012, pp.~592--600.

\bibitem{Keynes}
{\sc J.~M. Keynes}, {\em A Treatise on Probability}, Macmillan and Co., 1921.
\newblock Reprint BN Publishing, 2008.

\bibitem{IDS_HETEROSCEDASTIC_NOISE}
{\sc J.~Kirschner and A.~Krause}, {\em Information directed sampling and
  bandits with heteroscedastic noise}, Proceedings of Machine Learning
  Research,  (2018), pp.~1--28.

\bibitem{Knight}
{\sc F.~H. Knight}, {\em Risk, Uncertainty and Profit}, Houghton Mifflin, 1921.
\newblock reprint Dover 2006.

\bibitem{Lattimore_book}
{\sc T.~Lattimore and C.~Szepesv\'ari}, {\em Bandit Algorithms}, Cambridge
  University Press, 2019.

\bibitem{John_relaxed_control}
{\sc C.~Reisinger and Y.~Zhang}, {\em Regularity and stability of feedback
  relaxed control}, SIAM Journal on Control and Optimization,  (2021),
  pp.~3118--3151.

\bibitem{Convex_Analysis_book}
{\sc R.~Rockafellar}, {\em Convex Analysis}, Princeton university press, 1972.

\bibitem{Tsitsiklis_correlated_bandit}
{\sc P.~Rusmevichientong, A.~Mersereau, and J.~N. Tsitsiklis}, {\em {A
  Structured Multiarmed Bandit Problem and the Greedy Policy}}, in {Proceedings
  of the IEEE Conference on Decision and Control}, 2009.

\bibitem{Note_Gittins_UCB}
{\sc D.~Russo}, {\em A note on the equivalence of upper confidence bounds and
  gittins indices for patient agents}, Operations Research,  (2021),
  pp.~273--278.

\bibitem{Russo_Roy_Thompson}
{\sc D.~Russo and B.~V. Roy}, {\em An information-theoretic analysis of
  thompson sampling}, Journal of Machine Learning Research,  (2016), pp.~1--30.

\bibitem{IDS}
\leavevmode\vrule height 2pt depth -1.6pt width 23pt, {\em {Learning to
  Optimize via Information-Directed Sampling}}, Operations Research,  (2017),
  pp.~1--23.

\bibitem{Tutorial_on_Thompson_Sampling}
{\sc D.~Russo, B.~V. Roy, A.~Kazerouni, I.~Osband, and Z.~Wen}, {\em A tutorial
  on thompson sampling}, Foundations and Trends in Machine Learning, 11 (2018),
  pp.~1--96.

\bibitem{Knowledge_Gradient}
{\sc I.~O. Ryzhov, W.~B. Powell, and P.~I. Frazier}, {\em The knowledge
  gradient algorithm for a general class of online learning problems},
  Operations Research,  (2012).

\bibitem{BE_Alg}
{\sc S.~P. Singh, T.~Jaakkola, M.~L. Littman, and C.~Szepesv\'ari.}, {\em
  {Convergence results for single-stepon-policy reinforcement-learning
  algorithms}}, Machine Learning,  (2000).

\bibitem{Thompson_original}
{\sc W.~R. Thompson}, {\em {on the likelihood that one unknown probability
  exceeds another in view of the evidence of two samples}}, Biometrika,
  (1933), pp.~285--294.

\bibitem{eps_Greedy}
{\sc J.~Vermorel and M.~Mohri}, {\em Multi-armed bandit algorithms and
  empirical evaluation}, in European Conference on Machine Learning., 2005.

\bibitem{Relaxed_control}
{\sc H.~Wang, T.~Zariphopoulou, and X.~Zhou}, {\em Reinforcement learning in
  continuous time and space:a stochastic control approach}, Journal of Machine
  Learning Research,  (2020), pp.~1--34.

\bibitem{Gradient_flow_lukasz}
{\sc D.~{Šiška} and L.~Szpruch}, {\em {Gradient Flows for Regularized
  Stochastic Control Problems}}.
\newblock {arXiv:2006.05956}.

\end{thebibliography}

\paragraph{Acknowledgments.} Samuel Cohen also acknowledges the support of the Oxford-Man Institute for Quantitative Finance, and the UKRI Prosperity Partnership Scheme (FAIR) under the EPSRC Grant EP/V056883/1, and the Alan Turing Institute. Tanut Treetanthiploet
thanks the University of Oxford for research support while completing this work, and
acknowledges the support of the Development and Promotion of Science and Technology
Talents Project (DPST) of the Government of Thailand and the Alan Turing Institute.
 \end{document}